\documentclass[a4paper]{article}

\usepackage[english]{babel}
\usepackage[utf8x]{inputenc}
\usepackage[T1]{fontenc}
\usepackage{lipsum}
\usepackage{enumerate}
\usepackage{blindtext}
\usepackage{graphicx,amssymb,amsmath,textcomp}
\usepackage{circuitikz}
\usepackage[a4paper,top=2.5cm,bottom=2.5cm,left=3.0cm,right=3.0cm,marginparwidth=1.75cm]{geometry}
\usepackage{tikz,tikz-cd}
\usetikzlibrary{matrix,arrows,positioning,calc}
\usepackage{verbatim}
\usepackage{palatino, mathpazo}
\usepackage[nottoc]{tocbibind}
\usepackage{amsmath,amsthm}
\usepackage{faktor}
\usepackage{xfrac}
\usepackage{graphicx}
\usepackage[colorlinks=true, allcolors=purple]{hyperref}
\usepackage{mathtools}

\newcommand{\vdim}{{\rm vdim \ }}

\newcommand{\spec}{{\rm Spec}\,}

\newcommand{\Spec}{\mathrm{Spec}\,}
\newcommand{\cO}{\mathcal{O}}
\newcommand{\cL}{\mathcal{L}}

\newcommand{\LL}{\mathbb{L}}
\newcommand{\TT}{\mathbb{T}}
\newcommand{\Gm}{\mathbb{G}_m}

\newcommand{\N}{{\mathbb N}}
\newcommand{\Z}{{\mathbb Z}}

\newcommand{\K}{{\mathbb K}}
\newcommand{\tx}{\tilde{x}}
\newcommand{\ty}{\tilde{y}}

\newcommand{\dR}{{d_{dR}}}
\newcommand{\Map}{\mathrm{Map}}
\newcommand{\bfem}[1]{\textbf{\emph{#1}}}

\newtheorem{thmx}{Theorem}

\title{Shifted Contact Structures on \\ Exact Symplectic Fibrations}

\author{M. Fırat Arıkan\footnote{farikan@metu.edu.tr; Dept. of
  Mathematics, METU, 06800, Ankara, T\"urkiye}, \,K. İlker
  Berktav\footnote{berktav@metu.edu.tr; Dept. of Mathematics, METU,
  06800, Ankara, T\"urkiye}, \,Efe
  İzbudak\footnote{efe.izbudak@studium.uni-hamburg.de; Dept. of
Mathematics, Universität Hamburg, 20146, Hamburg, Germany}}

\date{\vspace{-5ex}}

\begin{document}

%%%%%%%%%%%%%%%%%%%% Text italic %%%%%%%%%%%%%%%%%%%%%%%%%%%%
\theoremstyle{plain}
\newtheorem{theorem}{Theorem}[section]
\newtheorem{lemma}[theorem]{Lemma}
\newtheorem{proposition}[theorem]{Proposition}
\newtheorem{corollary}[theorem]{Corollary}

%%%%%%%%%%%%%%%%%%%% Text roman %%%%%%%%%%%%%%%%%%%%%%%%%%%%%
\theoremstyle{definition}
\newtheorem{notations}[theorem]{Notations}
\newtheorem{notation}[theorem]{Notation}
\newtheorem{remark}[theorem]{Remark}
\newtheorem{observation}[theorem]{Observation}
\newtheorem{definition}[theorem]{Definition}
\newtheorem{condition}[theorem]{Condition}
\newtheorem{construction}[theorem]{Construction}
\newtheorem{example}[theorem]{Example}
\newtheorem{claim}[theorem]{Claim}

\let\pf\proof
\let\epf\endproof
\numberwithin{equation}{section}

\maketitle

\begin{abstract}
  Within the framework of classical contact geometry, the first
  author introduced the notion of \textit{contact symplectic
  structure}-- the data of an exact symplectic fibration with contact
  base-- and proved under mild conditions that the total space of
  such a fibration inherits a contact structure compatible with the
  fibration map, thereby establishing the contact Thurston theorem.
  This paper provides a derived contact version of that result by
  incorporating our prior work on the derived symplectic
  Thurston theorem.

  In this paper, we prove, under certain conditions, that if a
  morphism  $\pi: X \rightarrow S$ of derived stacks has a shifted
  exact symplectic fibration structure
  and the target stack $S$ admits a shifted contact structure, then
  one can construct a shifted contact structure on the source stack
  $X$, compatible with $\pi$ in a sense similar to the smooth case.
  Our framework relies on the theory of relative shifted structures;
  hence our result, called the \textit{derived contact Thurston
  theorem}, in fact establishes a relative-to-absolute type construction.

  As an application, we present examples of our relative-to-absolute
  construction formalism in the derived contact setting, including
  conormal stacks, quotient mapping stacks, and affine exact
  symplectic fibrations.

\end{abstract}

\tableofcontents

%===================================
%===================================
%===================================

\section{Introduction and summary}

In classical symplectic geometry, a result attributed to Thurston
(see \cite[Lemma 6.2]{Mcduff}) proposes that if a \bfem{symplectic
fibration} $\pi: E\to B$ \emph{(meaning a locally trivial fiber
    bundle with symplectic fibers whose structure group lies within the
group of symplectomorphisms of the fiber)} with connected symplectic
base satisfies some mild condition, then one can build a symplectic
form on the total space $E$ compatible with the fibration map $\pi$.
Recently, a smooth contact analog of this construction was given by
the first author in \cite{mfa}, where it was shown that the total
space of a compact \bfem{contact symplectic fibration}
\textit{(meaning a symplectic fibration with exact symplectic fibers
and a contact base)} can be endowed with a contact structure
compatible with the fibration map.

More recently, in the context of \textit{derived symplectic geometry}
(DSG), the authors proved in \cite{abi} the derived version of
Thurston's result (called the \textit{derived --symplectic-- Thurston
theorem}), providing a method for constructing a compatible
(absolute) shifted symplectic structure on the source $X$ of a
shifted symplectic fibration \( \pi: X \rightarrow S \) with a
shifted symplectic target. Here the (shifted) symplectic fibration
data consists of a \textit{relative shifted symplectic $2$-form} on
the fibration morphism, explaining why we emphasize the
relative-absolute dictionary. The prior work \cite{abi} therefore
describes a relative-to-absolute construction for the derived
symplectic setting.

Inspired by the above results, one naturally asks if this
relative-to-absolute construction formalism of the prior work
\cite{abi} can also be applied to obtain a derived contact analog of
\cite{mfa}. In this regard, we adapt the terminology and methods of
\cite{abi} to \textit{derived contact geometry} (DCG) and extend the
theory of contact symplectic fibrations to the derived setting,
thereby investigating further consequences and applications.

\paragraph{Results of the paper.}
To present and confirm the results in the general DCG setting,
similar to those in the smooth contact setup, we use or adapt the
ideas and tools from the DSG setting discussed in \cite{abi}.
For instance, an \textit{$n$-shifted (twisted) exact symplectic
fibration structure} (cf. Definition \ref{defn_shifted symp
fibration}) replaces the notion of \textit{$n$-shifted symplectic
fibration structure} from \cite{abi}. %Indeed, one can easily verify
% the following result, which is just the DCG analogue (needed for
% the present paper) of the \textit{induced compatible symplectic
% fibration theorem} (Theorem 1.1) of \cite{abi} established in the DSG setting.

In brief, we aim to extend the results of \cite{mfa} to the DCG
setting.  Our formulation builds on the natural
\textit{symplectic-contact dictionary} and incorporates the approach
of our prior work \cite{abi}.
\medskip

Let us summarize our results. Denote by $\mathbb{K}$ an algebraically
closed field of characteristic zero. Then a variation of
\cite[Theorem 1.1]{abi} with obvious modifications will give the
following immediate result:

\begin{thmx}[Induced compatible exact symplectic fibrations] \label{thm:2}
  Let $\pi: X \rightarrow S$ be a morphism of derived Artin
  $\K$-stacks. If $X$ admits an $n$-shifted contact form $\alpha_X$
  whose contact line bundle is pulled back from the base, $\cL_X
  \simeq \pi^*\cL$ for a line bundle $\cL$ on $S$, and the
  restriction of $d_{tot}\alpha_X$ to each geometric fiber $X_s$ is
  non-degenerate --- equivalently, defines an $n$-shifted symplectic
  structure on $X_s$ after any trivialization of the line
  $\iota_s^*\pi^*\cL$ --- then $\pi$ admits an $\cL$-twisted
  $n$-shifted exact symplectic fibration structure such that
  $\alpha_X$ is compatible with $\pi$ in the sense of Definition
  \ref{def:compatibility}.    (See Section \ref{sec:Ind_symp_fib}.)
\end{thmx}
This result can indeed be verified easily, as it amounts to the DCG
analogue (needed for the present paper) of the \textit{induced
compatible symplectic fibration theorem} (Theorem 1.1) of \cite{abi}
established in the DSG setting.
\medskip

Conversely, inspired by the result of \cite{mfa} (given in the smooth
setting), one could ask the following: \textit{Does there exist a
  compatible (absolute) shifted contact structure on the source $X$
  induced by the given relative shifted exact symplectic structure on a
  morphism $\pi:X\rightarrow S$ of derived stacks,  with a shifted
contact target $S$?} The affirmative answer is given below and can be
considered a DCG analog of the derived Thurston theorem.

\begin{thmx}[Derived Contact Thurston Theorem] \label{thm:1}
  Let $n\in \Z$. Suppose that $X,S$ are two locally finitely
  presented derived $\K$-stacks and that $\pi: X \rightarrow
  (S,\sigma, \cL_S[n])$ is a morphism of derived stacks, with
  $n$-shifted contact target, such that $\pi$ carries an $\cL_S$-twisted $n$-shifted
  exact symplectic fibration structure in the sense of  Definition
  \ref{defn_shifted symp fibration}.
  Then one can construct, under certain conditions, a compatible
  $n$-shifted contact structure on $X$ (cf. Theorem \ref{thm:
  proof_thurston in dcg}).
\end{thmx}

As applications of Theorem \ref{thm:1}, we present in this paper
several constructions of absolute shifted contact structures on
\textit{certain} exact symplectic fibrations, including conormal stacks and
(a dense Zariski-open derived substack of) the quotient mapping stacks. More
precisely, we prove:
\begin{thmx}\label{cor: app to Thm1}
  \begin{enumerate}[\upshape(a)]
    \item If $f: X \rightarrow S$ is a morphism of derived Artin
      $\K$-stacks locally of finite presentation such that $S$ admits
      an $(n-1)$-shifted contact structure with contact line bundle
      $\cL_S$ and $(n-1)$-contact form $\sigma$, then the
      $\cL_S$-twisted $n$-shifted conormal stack $\pi:
      N^*_{\cL_S}[n] f\longrightarrow S$ carries a canonical
      $\cL_S$-twisted $(n-1)$-shifted exact symplectic fibration
      structure $\omega$.
      
      Assume further that there exists a
      $\pi^*\cL_S$-twisted $(n-1)$-shifted 1-form $\Omega \in
      \mathcal{A}^{1}(N^*_{\cL_S}[n] f, \pi^*\cL_S, n-1)$ with
      $\iota_s^*d_{tot}\Omega \sim \omega_s := \iota_s^*\omega$ for
      every $s \in S(\K)$; and that after any trivialization of the line $L_s
      := s^*\cL_S$, $\omega_s$ is the canonical $(n-1)$-shifted exact
      symplectic structure of the fiber. Then there exists an induced $(n-1)$-shifted contact structure,
      with contact line bundle the restriction of $\pi^*\cL_S$, on a
      dense Zariski-open derived substack of $N^*_{\cL_S}[n] f$,
      compatible with the twisted exact symplectic fibration $\pi$
      (cf. Proposition \ref{prop: conormal}).

      Furthermore, the result extends   to the case of $\beta$-deformed conormal stacks for
      twisted closed 1-forms $\beta \in \mathcal{A}^{1,cl}(S, \cL_S,
      n+1)$ (cf. Corollary \ref{cor: conormal}).

    \item  Let $Y$ be an $n$-shifted contact derived Artin $\K$-stack
      and $X$ be a smooth projective Calabi-Yau $m$-fold. Suppose $f:
      L \rightarrow Y$ is an $n$-shifted Legendrian morphism and $p:
      Y \rightarrow S$ is an $n$-shifted Legendrian fibration such
      that the base $S$ admits an $(n-1)$-shifted contact structure
      $(\sigma_S, \cL_S[n-1])$. Under certain conditions, there
      exists an $(n-m-1)$-shifted contact structure on a dense
      Zariski-open derived substack of the quotient mapping stack
      $L_X := [\mathsf{Map}(X,\widetilde{L})/\mathbb{G}_m],$ 
      compatible with the induced twisted exact symplectic fibration
      $\Pi: L_X \to S_X :=
      [\mathsf{Map}(X,\widetilde{S})/\mathbb{G}_m]$ (cf. Proposition
      \ref{prop: mapping_stack_thurston}).
  \end{enumerate}
\end{thmx}
Finally, we establish a canonical construction of a shifted contact
structure on the source for an affine $n$-shifted exact symplectic
fibration, with contact base.
In brief, we have:

\begin{thmx}[Affine exact symplectic fibrations with contact base]
  \label{thm:affine_model}
  There is a canonical affine  model for an $n$-shifted exact
  symplectic fibration structure on the map $\spec\beta: \spec A \to
  \spec B $ of derived affine schemes over an $n$-shifted contact
  base, induced from a submersion $\beta:B\rightarrow A$ of standard
  form cdgas, such that $\spec A$ inherits a natural shifted contact
  structure (cf. Section \ref{sec:affine symplectic fibration}).
\end{thmx}

\paragraph{Organization of the paper.} In Section
\ref{sec:terminology}, we overview the frameworks of DSG and DCG.
Section \ref{sec:Relative structure} outlines the relative setting.
We prove Theorem \ref{thm:2} in Section \ref{sec:Ind_symp_fib}. The
precise version of Theorem \ref{thm:1} (cf. Theorem \ref{thm:
proof_thurston in dcg}) will be established in Section
\ref{sec:results}.  We discuss, in Section \ref{sec:applications},
applications of Theorem \ref{thm:1} and give the proof of Theorem
\ref{cor: app to Thm1}.  Section \ref{sec:affine symplectic
fibration} concludes the paper by investigating the affine case and
establishing Theorem \ref{thm:affine_model}. Lastly, in addition to
the main text, the paper also includes an appendix, Appendix
\ref{appendix_DAG}, covering the basics of derived algebraic geometry.

\paragraph{Conventions.} Throughout the paper, all cdgas will be
graded in non-positive degrees and over $ \mathbb{K}.$ We will always
consider $ \K $-schemes/stacks and assume that all classical $ \K
$-schemes are \emph{locally of finite type}, and that all derived $
\K $-schemes/stacks $ {X} $ are  \emph{locally finitely presented.}

\paragraph{Acknowledgments.}
This work was supported by the Scientific and Technological Research
Council of Türkiye (TÜBİTAK) under the 1001-Scientific and
Technological Research Projects Funding Program (Project No: 125F117).

All authors warmly thank the \textit{Higher Structures Research
Group} in the Mathematics Department at
Middle East Technical University for creating a stimulating research
environment.

%===================================
%===================================
%===================================
%\newpage

\section{Preliminaries: Derived contact and symplectic geometry}
\label{sec:terminology}

We outline in this section the foundational concepts from derived
symplectic and contact geometry, following
Pantev--To\"{e}n--Vaqui\'{e}--Vezzosi (PTVV) \cite{PTVV},
Brav--Bussi--Joyce \cite{Brav}, and Berktav \cite{kib1}.

\medskip

Let $X$ be a derived Artin $\K$-stack. Denote by $\LL_X, \TT_X,
\,DR(X)$ the \bfem{cotangent complex}, the \bfem{tangent complex},
and the \bfem{de Rham complex} of $X$, respectively. Appendix
\ref{appendix_DAG} details derived algebraic geometry basics. By a
\bfem{$p$-form of degree $n$ on $X$}, we mean an $n$-cocycle in $DR^p(X) [p]$.
Likewise,  by a \bfem{closed $p$-form of degree $n$ on $X$}, we mean
an $n$-cocycle in $\prod_{i \geq p} DR^i(X) [p]$ for the total
differential $d_{tot}=d+ \dR$. Denote the spaces of such forms by
$\mathcal{A}^p(X,n)$ and $\mathcal{A}^{p,cl}(X,n)$, respectively.

Note  that a $1$-form $\lambda \in \mathcal{A}^1(X,n)$ is $d$-closed of
weight $1$, so that \[d_{tot}\lambda = d_{dR}\lambda \ \text{ and } \
(d_{dR}\lambda, 0, 0, \dots) \in \mathcal{A}^{2,cl}(X,n).\] This is the
sense in which $d_{tot}$ is applied to $1$-forms throughout the paper.

\medskip

Also, we define an \bfem{equivalence relation} by considering paths
between objects in $\mathcal{A}^p(X,n)$ and $\mathcal{A}^{p ,
cl}(X,n)$, respectively. A \bfem{path} from $\omega^1$ to $\omega^2$
in $\mathcal{A}^p(X,n)$ is given by an element $h$ of weight $ p $
and degree $ (p + n) -1 $ such that $\omega^1-\omega^2=dh.$
Similarly, a path from $\omega^1=\omega^1_p + \omega^1_{p+1} + \cdots
$ to $\omega^2=\omega^2_p + \omega^2_{p+1} + \cdots $ is given by a
formal series $h=h_p + h_{p+1} + \cdots $ such that we have
$\omega^1_q -\omega^2_q= dh_q + \dR h_{q-1}.$

\begin{definition}[Shifted Symplectic and Contact Structures]
  Let $X$ be a derived Artin $\K$-stack.
\begin{enumerate}\itemsep=.1in
    \item An \bfem{$n$-shifted symplectic structure} on $X$ is a
      closed 2-form $\omega \in \mathcal{A}^{2,cl}(X, n)$ whose
      underlying map $\omega^\flat\colon \TT_X \to \LL_X[n]$ is an
      equivalence (the \emph{non-degeneracy condition}) in the stable
      $\infty$-category $QCoh(X)$.  Denote the space of such by
      $\mathsf{Symp}(X,n)$.

      If, in addition, the   form $\omega$ is $d_{tot}$-exact, we
      call $\omega$ an $n$-shifted \bfem{exact} symplectic structure.
    \item An \bfem{$n$-shifted contact structure} on $X$ consists of
      a (twisting) line bundle $\mathcal{L}_X$, a twisted $n$-shifted
      1-form $\alpha_X \in \pi_0 \mathbb{R}\Gamma(X, \LL_X
      \otimes^{\mathbb{L}}_{\cO_X} \cL_X[n])$, and a non-degenerate
      contact distribution defined as the homotopy fiber
      $$\mathcal{K}_X \simeq \mathrm{fib}(\alpha_X^\vee\colon \TT_X
      \to \mathcal{L}_X[n])$$ such that the restriction of the
      derived de Rham differential $d_{dR}\alpha_X$ to the contact
      distribution $\mathcal{K}_X$ induces an equivalence
      $\mathcal{K}_X \xrightarrow{\sim} \mathcal{K}_X^\vee
      \otimes^{\mathbb{L}}_{\cO_X} \mathcal{L}_X[n]$ in   $QCoh(X)$.
      We denote the space of such structures by $\mathsf{Cont}(X,n)$.
      A typical element of $\mathsf{Cont}(X,n)$ is written as a tuple
      $(X,\alpha_X, \cL_X[n])$ or simply as $(X,\alpha_X)$. We also
      call such defining  1-forms simply \bfem{$n$-contact forms.}
  \end{enumerate}
\end{definition}

\begin{notation}[Twisted forms and weights] \label{notn:twisted}
  Let $\mathcal{L}$ be a line bundle on $X$. We write
  \[ \mathcal{A}^p(X, \mathcal{L}, n) := \pi_0 \mathbb{R}\Gamma\big(X,
  \wedge^p \LL_X \otimes^{\mathbb{L}}_{\cO_X} \mathcal{L}[n]\big) \]
  for the space of \bfem{$\mathcal{L}$-twisted $p$-forms of degree
  $n$}, and $\mathcal{A}^{p,cl}(X,\mathcal{L},n)$ for the closed ones,
  so that $\mathcal{A}^p(X,\cO_X,n)=\mathcal{A}^p(X,n)$ and an
  $n$-contact form on $X$ is then an element of $\mathcal{A}^1(X,
  \mathcal{L}, n)$. Write $p_X: \widetilde{X} \rightarrow X$ for the
  principal $\mathbb{G}_m$-bundle of $\mathcal{L}$, with the action
  map $\rho$. Since $p_{X*}
  \cO_{\widetilde{X}} \simeq \bigoplus_{w \in \Z} \mathcal{L}^{\otimes
  w}$ as $\mathbb{G}_m$-representations, pullback along $p_X$
  identifies $\mathcal{A}^p(X, \mathcal{L}, n)$ with the space of
  \bfem{basic} (i.e. lying in the image of $p_X^*\LL_X$)
  \bfem{weight}\footnote{The $\Gm$-action has \bfem{weight 1} with respect to
    $\beta_{\widetilde{X}}\in \mathcal{A}^p(X,n)$ if the derived
    pullback of the action map
    $\rho: \Gm \times \widetilde{X} \to \widetilde{X}$ satisfies $\rho^*
    \beta_{\widetilde{X}} \simeq z \cdot \pi_{\widetilde{X}}^*
    \beta_{\widetilde{X}}$,
    where $z$ is the character coordinate on $\Gm$
  \cite[Section 2.1.1]{Calaque2019}.} $1$
  $p$-forms of degree $n$ on $\widetilde{X}$, and likewise for closed
  forms. We refer to this as the \bfem{weight dictionary}.
\end{notation}

\begin{remark} \label{rmk:twisted d_dR}
  For a twisted $1$-form $\alpha_X \in \mathcal{A}^1(X,\mathcal{L}_X,n)$
  the expression $d_{dR}\alpha_X$ is not defined on $X$ without a
  choice of connection on $\mathcal{L}_X$; only its restriction to
  $\mathcal{K}_X$ is, which is what the non-degeneracy condition above
  refers to. Whenever we write $d_{dR}\alpha_X$ or $d_{tot}\alpha_X$
  for a twisted form, we mean the corresponding weight $1$ form on the
  symplectification, computed through the weight dictionary in Notation \ref{notn:twisted}. When
  $\mathcal{L}_X \simeq \cO_X$ we call the contact structure
  \bfem{coorientable}; a trivialization then identifies twisted forms
  with ordinary forms and every statement below reduces to its
  untwisted form. For the non-coorientable case, $\mathcal{L}_X
  \simeq \cO_X$ holds only locally on $X$, where  $\mathcal{L}_X$ is
  trivialized. In that situation, $\alpha_X$ is identified locally
  with a genuine $1$-form of degree $n$ on $X$. Therefore, we may
  interchangeably use either local $1$-forms or twisted $1$-forms
  when indicating the contact data.
\end{remark}

\begin{lemma}[Fiberwise twists] \label{lem:fiberwise twist}
  Let $\pi: X \rightarrow S$ be a morphism of derived Artin
  $\K$-stacks, $\mathcal{L}$ a line bundle on $S$, and $\iota_s: X_s
  \hookrightarrow X$ the inclusion of the geometric fiber over $s \in
  S(\K)$. Then there is a canonical equivalence
  $\iota_s^*\pi^*\mathcal{L} \simeq \cO_{X_s} \otimes_{\K} L_s$,
  where $L_s := s^*\mathcal{L}$ is a $1$-dimensional $\K$-vector
  space. Consequently, for all $p$ and $n$,
  \[ \mathcal{A}^{p}(X_s, \iota_s^*\pi^*\mathcal{L}, n) \simeq
  \mathcal{A}^{p}(X_s, n) \otimes_{\K} L_s, \qquad
  \mathcal{A}^{p,cl}(X_s, \iota_s^*\pi^*\mathcal{L}, n) \simeq
  \mathcal{A}^{p,cl}(X_s, n) \otimes_{\K} L_s, \]
  so a choice of basis $\ell_s$ of $L_s$ identifies twisted forms on
  $X_s$ with ordinary ones. A twisted $2$-form on $X_s$ is
  non-degenerate if and only if the ordinary $2$-form obtained from
  one (equivalently, any) basis of $L_s$ is non-degenerate, since two
  bases differ by a unit in $\K^\times$.
\end{lemma}

\begin{proof}
  The Cartesian square defining $X_s$ gives $\pi \circ \iota_s \simeq
  s \circ p_s$, where $p_s: X_s \rightarrow \spec \K$ is the
  structure morphism; hence $\iota_s^*\pi^*\mathcal{L} \simeq
  p_s^*s^*\mathcal{L} \simeq \cO_{X_s} \otimes_{\K} L_s$. Since $L_s$
  is a $1$-dimensional $\K$-vector space, the functor $- \otimes_{\K}
  L_s$ is exact and commutes with $R\Gamma$ and with the de Rham
  differentials, which gives the displayed equivalences.
  Non-degeneracy of a twisted $2$-form $\beta$ on $X_s$ means that
  $\beta^\flat: \mathbb{T}_{X_s} \rightarrow \mathbb{L}_{X_s} \otimes
  \iota_s^*\pi^*\mathcal{L}[n]$ is an equivalence; under the
  identification above, $\beta^\flat \simeq (\beta_0)^\flat \otimes
  \mathrm{id}_{L_s}$ for the ordinary $2$-form $\beta_0$ associated
  to a basis, and changing the basis rescales $\beta_0$ by a unit,
  which does not affect being an equivalence.
\end{proof}

\begin{definition}[Lagrangian vs. Legendrian structures]
  \label{defn_isotropic Lag}
  \begin{enumerate}
    \item[ ]
    \item Let $f:  {L} \to X$ be a morphism of derived Artin stacks,
      with a target $X$ carrying a symplectic structure
      $\omega_{X}\in\mathsf{Symp}(X,n)$. By an
      \emph{\textbf{$n$-shifted isotropic structure}} on $f$
      (relative to $ \omega_{X} $), we mean a path $h_{\mathcal{L}}$
      from $0$ to $f^*(\omega_{X})$ in  $\mathcal{A}^{2, cl} ( {L}, n)$.

      An \bfem{$n$-shifted Lagrangian structure} on  $ f $  is
      defined to be an $n$-shifted isotropic structure such that the
      sequence $\mathbb{T}_{{L}}\rightarrow f^*(\mathbb{T}_{{X}})
      \rightarrow \mathbb{L}_{{L}}[n]$ is a homotopy fiber
      sequence\footnote{Equivalently, we can require the  map
        $\chi_{h_{{L}}}:     \mathbb{T}_{{L}/X} \rightarrow
      \mathbb{L}_{{L}}[n-1]$ to be an equivalence.}.  In that case,
      we  say that ${L}$ is \bfem{Lagrangian in $ ({X}, \omega_{X}).$}

    \item A morphism $f: L \to X$ into an $n$-shifted contact stack
      is called \bfem{Legendrian} if it is equipped with a
      null-homotopy $f^*\alpha_X \simeq 0$ which induces a stable
      fiber sequence\footnote{Or, equivalenty, the fiber sequence
        $\mathbb{T}_{{L}} \to f^*(\mathcal{K}) \rightarrow
        \mathbb{L}_{{L}}[n]$. For the equivalence, see
      \cite{IzbudakBerktav2026}. }
      \begin{equation}
        \mathcal{O}_L \longrightarrow \mathbb{L}_{L/X}
        \otimes_{\mathcal{O}_L} f^*\mathcal{L}_X[n-1] \longrightarrow
        \mathbb{T}_L
      \end{equation}
      in the stable $\infty$-category $QCoh(L)$.
  \end{enumerate}

\end{definition}
The theory of shifted contact structures forms the odd-dimensional
counterpart to shifted symplectic stacks, with the line bundle
accounting for the scaling ambiguity in contact geometry. Let us
present a prototype construction and provide a list of foundational
results established so far in the setting of derived contact
geometry, see \cite{kib1,kib2,kib3}. For the following exposition, we
follow \cite{Berktav_indam}.

%\subsection{Recollection of derived contact geometry (DCG) }
% We first briefly recall the definition of shifted contact
% structures and mention the basics of the theory. Afterward, we
% discuss more recent work and several applications to moduli theory
% and physics.

%\subsection{Foundations for shifted contact structures}

%Let us present a prototype construction and provide a list of
% foundational results established so far in the setting of derived
% contact geometry, see \cite{kib1,kib2,kib3}. For the following
% exposition, we follow \cite{Berktav_indam}.

\begin{example}[Prototype Construction] \label{model example}
  Let $n=-2\ell -1$ for  $\ell\in \N$. There exists a canonical
  (strict) {$n$-shifted contact structure} on $X:=\spec A$ induced by
  $\ker \alpha$,  where the pair $(A,\alpha)$ is defined\footnote{The
    example is in fact a variation on the symplectic case \cite[Example
  5.8]{Brav}.} below. See \cite[Example 3.10]{kib1} or \cite[Example
  2.21]{kib2} for details.

  \begin{enumerate}
    \item  As a \bfem{commutative graded algebra}, $A$ is  defined to
      be the free graded $\K$-algebra over $A(0)$ generated by the variables
      \begin{align}
        & x_1^{-i}, x_2^{-i}, \dots, x_{m_i}^{-i}& &\text{ in degree
        } -i \ \ \ \text{ for } i= 1,  \dots, \ell, \nonumber  \\
        & y_1^{n+i}, y_2^{n+i}, \dots, y_{m_i}^{n+i}& & \text{ in
        degree } n+i \ \text{ for } i=1,\dots, \ell,  \nonumber \\
        & z^n, y_1^{n}, y_2^{n}, \dots, y_{m_0}^{n}& & \text{ in degree } n.
      \end{align} Here, $m_1,\dots, m_{\ell}$ are some non-negative
      integers and $ A(0) $ is a smooth $ \K $-algebra  of dimension
      $ m_0 $, with degree 0 variables $x_1^0, x_2^0, \dots,
      x_{m_0}^0$ in $A(0)$ defining global \'{e}tale coordinates
      $(x_1^0, x_2^0, \dots, x_{m_0}^0): \spec A(0) \rightarrow
      \mathbb{A}^{m_0}$ on $\spec A(0)$ such that
      $d_{dR}x_1^0,\dots,d_{dR}x_{m_0}^0$ form a $A(0)$-basis for
      $\Omega_{A(0)}^1$.

    \item The \bfem{internal differential} $d$ on $A$ is determined
      by  the equations
      \begin{align} \label{defn_internal d contact}
        d|_{A(0)}&=0; \ dx_j^{-i} =  \dfrac{\partial H}{\partial
        y_j^{n+i}} \text{ for all } i>0,j; \  \ dy_j^{n+i} =
        \dfrac{\partial H}{\partial x_j^{-i}} \text{ for all } i,j;
        \text{ and } \nonumber \\ -ndz^n&= H+d\Big[\sum_{i,j}
        (-1)^{i} ix_j^{-i} y_j^{n+i} \Big],
      \end{align}    where  $H\in A^{n+1}$ is the
      \emph{Hamiltonian}\footnote{An element in $A^{n+1}$ satisfying
        the  \emph{classical master equation} $ \displaystyle
        \sum_{i=1}^{\ell} \sum_{j=1}^{m_i} \dfrac{\partial H}{\partial
      x_j^{-i}} \dfrac{\partial H}{\partial y_j^{n+i}}=0$.}. Note
      that the condition on $H$ implies  $d^2=0$ on each generator
      \cite{Brav}. Counting generators degreewise, the contributions
      of $x_j^{-i}$ and $y_j^{n+i}$ cancel for $i=1,\dots,\ell$, as
      do those of $x_j^0$ and $y_j^n$, so the distinguished variable
      $z^n$ gives $\vdim(A)=-1$.%$ \displaystyle \sum_{i=1}^{\ell}
      % \sum_{j=1}^{m_i} \dfrac{\partial H}{\partial x_j^{-i}}
      % \dfrac{\partial H}{\partial y_j^{n+i}}=0 \text{ in } A^{n+2}.
      % $ %The condition on $H$ implies  $d^2=0$ on each generator
      % \cite{Brav}. Also, we have  $\vdim(A)=-1$.

    \item  The  \bfem{$n$-contact form} $\alpha \in \Omega_{A}^1[n]$ is given by
      \begin{equation} \label{local contact model}
        \alpha= d_{dR}z^n+ \displaystyle \sum_{i=0}^{\ell}
        \sum_{j=1}^{m_i} y_j^{n+i}d_{dR}x_j^{-i}.
      \end{equation}
  \end{enumerate}
\end{example}
Let us now outline below some
initial results, along with recent/ongoing works, for contact derived
stacks:%recent/ongoing works on shifted contact structures.
\paragraph{Darboux models,  symplectifications, and fancy examples.}
\cite{kib1} presents the contact variation of the result of
Brav-Bussi-Joyce \cite[Theorem 5.18]{Brav}, leading to Darboux-type
local models for shifted contact derived schemes. More precisely, we have:
\begin{theorem}[Derived contact Darboux theorem]\label{contact darboux}
  Every negatively shifted contact (locally finitely presented)
  derived $\K$-scheme $X$  is locally equivalent to $(\spec A,
  \alpha)$ for $A$ a (minimal) standard form cdga and $\alpha$ in
  standard contact Darboux form. See \cite[Theorem 3.13]{kib1}
\end{theorem}Note that if the shift is odd, the data $(A, d, \alpha)$
can be explicitly described by the  graded variables as in Example
\ref{model example}.

%Let $ X$ be a (locally finitely presented) derived $\K$-scheme
% carrying an $n$-shifted contact structure,  with $n<0$. Then we
% define: \begin{definition}
%The \emph{symplectification} $ \mathcal{S}_{{ {X}}} $ of $X$ is
% defined to be the total space $\widetilde{ L}$ of the
% $\mathbb{G}_m$-bundle of $L$ over $X$.
%\end{definition}
\begin{theorem}[Symplectifications] \label{thm_Symplectization}
  Let $ X$ be a (locally finitely presented) derived $\K$-scheme
  carrying an $n$-shifted contact structure,  with $n<0$. The
  \emph{symplectification} $ \mathcal{S}_{{ {X}}} $ of $X$ is defined
  to be the total space $\widetilde{ L}$ of the $\mathbb{G}_m$-bundle
  of $L$ over $X$. $\mathcal{S}_{{X}}$ is a derived stack equipped
  with an $n$-shifted symplectic form $  \omega_{{X}},$ which is
  canonically determined by the shifted contact structure of $X$, and
  for which the $\mathbb{G}_m$-action is of weight $1$. See
  \cite[Theorem 4.7]{kib1}.
\end{theorem}

It has been shown in  \cite{kib2} that Theorem \ref{contact darboux}
and Theorem \ref{thm_Symplectization} still hold for negatively
shifted contact derived Artin $\K$-stacks  locally of finite presentation.

In addition to standard Darboux constructions, one can construct
interesting examples of contact derived stacks within the framework
of Safronov's shifted geometric quantization \cite{Safronov2023}. In
brief,  \cite[Theorem 1.3]{kib2} establishes:
\begin{theorem}[Further examples]
  \begin{enumerate}
    \item The space $ {J^1[n]{X}:=T^*[n]}{X }\times \mathbb{G}_a[n]$,
      called the \emph{$n$-shifted 1-jet stack} of $X$, carries an
      $n$-shifted contact structure.
    \item Under certain conditions, there is a  $\mathbb{G}_m$-bundle
      on $ {T^*}{X}$ carrying a 0-shifted contact structure.
    \item Under certain conditions, there is a  $\mathbb{G}_m$-bundle
      on the $ c_1(\mathcal{G})$-twisted cotangent stack
      ${T^*}_{c_1(\mathcal{G})}X$ that carries a 0-shifted contact
      structure, where the element $ c_1(\mathcal{G}) \in
      \mathcal{A}^{1,cl}({X}, 1) $ denotes the \emph{characteristic
      class} of a 0-gerbe $\mathcal{G}$  on $X$.
    \item Assume that $ G $ is a simple algebraic group
      over  $ \mathbb{K} $, and $ C $ be a  smooth and proper
      curve$/\mathbb{K} $. Consider the derived moduli stacks
      $LocSys_G(C), Bun_G(C)$ of \emph{flat $G$-connections on $C$},
      \emph{principal $G$-bundles on $C$}. Then there is a
      $\mathbb{G}_m$-bundle on $ LocSys_G(C) $ with a 0-shifted
      contact structure.

  \end{enumerate}
\end{theorem}
Regarding the properties of Legendrians, on the other hand,
\cite{kib3} provides a \emph{Legendrian-Darboux-type} theorem. %the
% Legendrians in contact derived schemes.
\cite{kib3} also shows that the zero section morphism $j: {X}
\rightarrow {J^1[n]}{X}$  carries a  Legendrian structure, extending
the well-known classical result. %\footnote{Notice that this result
% is analogous to the one in classical contact geometry.}.

%\paragraph{The case of smooth stacks.}
In the case of smooth stacks, on the other hand, Maglio,  Tortorella
and Vitagliano \cite{Maglio2024} have recently developed and examined
\emph{$0$-shifted} and \emph{$+1$-shifted contact structures} on
\emph{differentiable stacks}, offering very interesting results and
constructions within the smooth-stacky setting.

\vspace{-0.1in}

\paragraph{Recent developments on DCG and applications.}

In a series of papers
\cite{IzbudakBerktav2026,IzbudakBerktav2,IzbudakBerktav3,
Izbudak_Perverse}, İzbudak and Berktav have recently constructed
several interesting moduli spaces carrying shifted contact structures
and developed tools to establish results extending the classical
contact geometry. There are also results with no classical
counterparts, such as the \textit{Legendrian intersection theorem}
and the \textit{contact AKSZ construction}. In what follows, we
highlight some of these results.
\vspace{-0.1in}

\paragraph{\textit{a. Transversality and derived Legendrian
intersections.}} İzbudak and Berktav have adapted in
\cite{IzbudakBerktav2026} the traditional transversality lemma from
contact geometry to the derived setup and have shown the
\textit{derived Legendrian intersection theorem}, along with numerous
applications, such as the derived geometry of the discriminant loci
of 1-jet bundles and certain moduli problems including projective
Higgs bundles, $\ell$-adic local systems, and Lie 2-groups.

In brief, \cite[Theorem 1.1]{IzbudakBerktav2026} formalizes
\textit{derived symplectification/contactification} and proves that taking the
quotient of a homogeneous symplectic space descends the symplectic
data to a contact structure, avoiding
a transverse hypersurface. More precisely, using the geometry of $\Gm$-torsors
and symplectification, \cite{IzbudakBerktav2026} establishes the
following key result.

\begin{theorem}[Derived Analogue of Classical Transversality]
  \label{thm:derived transversality}
  Given an $n$-shifted symplectic derived Artin stack $(\widetilde{X},
  \omega_{\widetilde{X}})$
  equipped with a $\Gm$-action of weight 1,
  the stack quotient $[\widetilde{X} / \Gm]$ inherits an
  $n$-shifted contact structure.
\end{theorem}

Furthermore, DCG links Legendrian structures to Lagrangian structures
via $\Gm$-equivariant lifts along the symplectification projection.
In this regard, \cite[Theorem 1.2]{IzbudakBerktav2026} proves the
derived Legendrian intersection theorem via base change and an
$\infty$-categorical descent cube, along with a certain lifting
property arising from the symplectic-contact dictionary. More
precisely, we have:
%Lifting intersection data to the symplectic level yields:

\begin{theorem}[Legendrian Intersection Theorem] \label{thm:derived
  Legendrian intersection}
  Intersection of two Legendrians $f_1 \colon L_1 \to X$ and $f_2
  \colon L_2 \to X$ in an $n$-shifted contact derived Artin stack admits
  an $(n-1)$-shifted  contact structure.
\end{theorem}

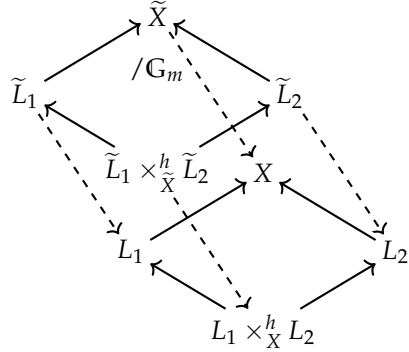
\begin{figure}[htbp]
  \centering
  \begin{tikzpicture}[scale=0.7]
    \node[fill=white, inner sep=2pt] (tZ) at (0,0) {$\widetilde{L}_1
    \times^h_{\widetilde{X}} \widetilde{L}_2$};
    \node[fill=white, inner sep=2pt] (tL1) at (-2.5, 1.5) {$\widetilde{L}_1$};
    \node[fill=white, inner sep=2pt] (tL2) at (2.5, 1.5) {$\widetilde{L}_2$};
    \node[fill=white, inner sep=2pt] (tX) at (0, 3) {$\widetilde{X}$};

    \node[fill=white, inner sep=2pt] (Z) at (2, -3) {$L_1 \times^h_X L_2$};
    \node[fill=white, inner sep=2pt] (L1) at (-0.5, -1.5) {$L_1$};
    \node[fill=white, inner sep=2pt] (L2) at (4.5, -1.5) {$L_2$};
    \node[fill=white, inner sep=2pt] (X) at (2, 0) {$X$};

    \draw[->, thick, dashed, preaction={draw, white, line width=4pt}]
    (tX) -- node[pos=0.3, left] {$/\Gm$} (X);
    \draw[->, thick] (tZ) -- (tL1);
    \draw[->, thick] (tZ) -- (tL2);
    \draw[->, thick] (tL1) -- (tX);
    \draw[->, thick] (tL2) -- (tX);

    \draw[->, thick, dashed] (tL1) -- (L1);
    \draw[->, thick, dashed] (tL2) -- (L2);

    \draw[->, thick, preaction={draw, white, line width=4pt}] (Z) -- (L1);
    \draw[->, thick, preaction={draw, white, line width=4pt}] (Z) -- (L2);
    \draw[->, thick, preaction={draw, white, line width=4pt}] (L1) -- (X);
    \draw[->, thick, preaction={draw, white, line width=4pt}] (L2) -- (X);
    \draw[->, thick, dashed] (tZ) -- node[pos=0.7, left] {$ $} (Z);
  \end{tikzpicture}
  \caption{\small The homotopy limit diagram. The upper plane is the
    derived intersection of the $\Gm$-equivariant Lagrangian lifts in
    the derived symplectification $\widetilde{X}$. The dashed arrows
    denote the descent via the principal $\Gm$-bundle quotient, giving
  the derived Legendrian intersection in the contact base $X$.}
\end{figure}
It is important to note that \textit{this lifting property translates
  intersection problems in contact moduli to equivariant intersection
problems in symplectic moduli}. This establishes the desired shifted
contact structure on the derived intersection of two Legendrian stacks.
\vspace{-0.1in}

\paragraph{\textit{b. Legendrian correspondences and Derived
Legendrian Category.}} The companion paper \cite{IzbudakBerktav2}
constructs the derived Legendrian category $\mathcal{F}_{c}(X)$ for
an $n$-shifted contact derived Artin stack $X$ and the
$(\infty,2)$-category $Leg_n$ of Legendrian correspondences in the
context of derived algebraic geometry, with several applications to
moduli theory.

Regarding applications, \cite[Corollary 1.3]{IzbudakBerktav2} in fact
shows that the derived Legendrian category $\mathcal{F}_{c}(X)$
encodes geometric and algebraic structures on certain derived moduli
spaces in the sense that the composition operations, endomorphism
algebras, and functoriality define the non-commutative and shifted
geometry of these spaces.

As ongoing work, our construction of the derived Legendrian category
$\mathcal{F}_{c}(X)$ and the formulation of topological cobordisms as
Legendrian spans will in fact serve as the necessary setup to define
surgery operations directly on contact stacks.
\vspace{-0.1in}

\paragraph{\textit{c. Contact AKSZ construction.}} The AKSZ
construction provides a formalism in derived algebraic geometry for
generating shifted geometric structures on derived mapping stacks.
Pantev, To\"{e}n, Vaqui\'{e}, and Vezzosi introduced this
construction for shifted symplectic structures and proved that the
derived mapping stack of a $d$-oriented stack into an $n$-shifted
symplectic stack admits an $(n-d)$-shifted symplectic structure
\cite[Theorem 2.1]{PTVV}. As a contact analogue of this result,
\cite{IzbudakBerktav3} establishes the AKSZ theorem for shifted
contact structures and its applications.

As one of key results, \cite[Theorem A]{IzbudakBerktav3} proves that
if $X$ is an $n$-shifted contact derived Artin stack and $Y$ is an
$\cO$-compact, $d$-oriented derived stack, the quotient mapping stack
$$[\Map(Y, \widetilde{X})/\Gm]$$ admits an $(n-d)$-shifted contact
structure. The companion paper \cite{IzbudakBerktav3}  also discusses
applications of this quotient mapping stack formalism, where the
derived analogues of specific topological field theories -- the
Jacobi, Courant-Jacobi, and Loop Space Sigma Models-- have been described.
\vspace{-0.1in}

\paragraph{\textit{d. Equivariant Contact Darboux Quotients.}}
İzbudak establishes in  \cite{Izbudak_Perverse} the
\textit{Equivariant Contact Darboux Theorem}, verifying that
$(-1)$-shifted contact stacks are locally geometric quotients of
derived discriminant loci. %This paper transports the equivariant
% Darboux theorem and the BBDJS perverse sheaf formalism into the
% context of derived contact geometry. By establishing the
% equivariant contact Darboux theorem, we prove that a $-1$-shifted
% contact stack admits a smooth derived Darboux atlas, which is
% locally equivalent to a geometric quotient stack in the presence of
% stabilizing groups in the smooth topology.
This approach is employed to linearize the derived Legendrian
categories and to define enumerative invariants, thereby connecting
contact moduli spaces to microlocal sheaf theory.
%===================================
%===================================
%===================================
\section{Relative structures}
\subsection{Twisted exact symplectic  fibrations} \label{sec:Relative structure}

The framework of PTVV's symplectic geometry \cite{PTVV} can be extended to the
relative setting, which is manifested by the morphisms of derived
stacks. The key is to adapt the de Rham algebra via the morphism
data. Given a morphism $\pi: X \rightarrow S$ of derived Artin
stacks, we can define the \bfem{relative de Rham algebra $DR(X/S)$}
as a graded commutative dg algebra
\begin{equation}
  DR(X/S)\simeq \Gamma(\mathrm{Sym}_{\mathcal{O}_X}(\mathbb{L}_{X/S}[-1])),
\end{equation} where $ \mathbb{L}_{X/S} $ is the \emph{relative
cotangent complex} defined via the cofiber sequence
$\pi^*\mathbb{L}_S \rightarrow \mathbb{L}_X \rightarrow \mathbb{L}_{X/S}.$
%When $\pi$ is a map of derived affine schemes $\spec A \rightarrow
% \spec B$ coming from a cofibration $B\rightarrow A$ of cofibrant
% cdgas, we have \begin{equation}
%DR(X/S)=DR(A/B) \simeq \mathrm{Sym}_A (\Omega^1_{A/B} [-1]).
%\end{equation}
Following \cite{CPTVV,Safronov2023}, we then have the following
relative concepts:
\begin{definition}

  \begin{enumerate}[\upshape(a)]
    \item   A \bfem{relative $p$-form of degree $n$ on $\pi: X
      \rightarrow S$} is a $d$-closed element of $DR(X/S)$ of weight
      $p$ and cohomological degree $p+n$.  Denote the space of such
      forms by $\mathcal{A}^p(X/S,n)$.

    \item     A \bfem{closed relative $p$-form of degree $n$ on $\pi:
      X \rightarrow S$} is a $(d+\dR)$-closed element
      $\omega_p+\omega_{p+1} + \cdots$ of $DR(X/S)$ such that
      $\omega_q$ is a $q$-form of degree $n-(q-p)$.  Denote the space
      of such forms by $\mathcal{A}^{p,cl}(X/S,n)$.
  \end{enumerate}

\end{definition}
Let us outline key properties of these relative concepts that extend
the absolute ones. Suppose $\pi: X \rightarrow S$ is a morphism of
derived Artin $\K$-stacks. We first notice that there exists a
null-homotopic sequence $DR(S) \xrightarrow{\pi^*} DR(X)
\xrightarrow{\cdot/S} DR(X/S)$ such that the quotient morphism $DR(X)
\rightarrow DR(X/S)$ factors as $$DR(X) \rightarrow DR(X \times S /S)
\rightarrow DR(X/S),$$ where the first map allows us to see forms on
$X$ as relative forms on the projection map $X \times S \rightarrow
S;$ and the second one is the pullback along $(id \times \pi): X
\rightarrow X \times S. $ We refer to \cite{Calaque2019} for details.

If, in addition, $\pi$ is locally of finite presentation, we can
define $\mathbb{T}_{X/S}= \mathbb{L}_{X/S}^{\vee}$ as the dual
(recall the fiber sequence $\mathbb{T}_{X/S} \rightarrow \mathbb{T}_X
\rightarrow \pi^* \mathbb{T}_S$). Then a relative 2-form $\omega$ of
degree $n$ induces a morphism
\begin{equation} \label{eqn:relative symplectic form}
  \omega^{\flat}: \mathbb{T}_{X/S} \rightarrow \mathbb{L}_{X/S} [n],
\end{equation}
yielding the following definition:

\begin{definition}\label{defn_relative symplectic strc}
  A \bfem{relative $n$-shifted symplectic structure $\omega$ on $\pi:
  X \rightarrow S$} is a closed relative $2$-form of degree $n$ on
  $\pi$ such that the induced map $\omega^{\flat}: \mathbb{T}_{X/S}
  \rightarrow \mathbb{L}_{X/S} [n]$ is an equivalence.
\end{definition}

\begin{remark}
  Assuming all stacks are $\K$-stacks, the relative objects on the
  morphism $X \rightarrow \mathrm{pt}$, with target as the point
  $\mathrm{pt}$, will recover the ordinary (absolute) shifted
  geometric structures defined previously.
\end{remark}

Next, we formulate the notion of an (exact) symplectic fibration in
the derived context. To this end, let $\pi: X \rightarrow S$ be a
morphism of derived Artin stacks. Define the \bfem{fiber $ X_s $ over
$s\in S$} by the derived fiber product diagram
\begin{equation}
  \begin{tikzcd}
    X_s:=X \times_{S} \{*\} \arrow[r] \arrow[d] & \{*\} \arrow[d, "s"] \\
    X \arrow[r,"\pi"]                                 & S,
  \end{tikzcd}
\end{equation}where we denote by $\iota_s$ the inclusion of the fiber
$X_s \hookrightarrow X$. Note that if $\omega$ is a relative
$n$-shifted symplectic structure on $\pi$, then, over $s\in S$, the
fiber $X_s$ has the induced $n$-shifted symplectic structure, written
$\omega_s$. Thus, we define:

\begin{definition}[Twisted exact symplectic fibrations]
  \label{defn_shifted symp fibration}
  Let $\mathcal{L}$ be a line bundle on $S$. An
  \bfem{$\mathcal{L}$-twisted $n$-shifted exact symplectic fibration
  structure} on a morphism $\pi: X \rightarrow S$ of derived Artin
  $\K$-stacks is the datum of
  \begin{itemize}
    \item a relative $\pi^*\mathcal{L}$-twisted
      $n$-shifted 1-form $\lambda \in \mathcal{A}^1(X/S,
      \pi^*\mathcal{L}, n)$ such that,
    \item the 2-form $\omega := d_{tot}\lambda =
      d_{dR}\lambda$ is a relative $n$-shifted symplectic structure on
      $\pi$ in the sense of Definition \ref{defn_relative symplectic
      strc}, with the evident twisted modification: the non-degeneracy
      now means that $\omega^\flat: \mathbb{T}_{X/S} \rightarrow
      \mathbb{L}_{X/S} \otimes \pi^*\mathcal{L}[n]$ is an
      equivalence. In particular, on each geometric fiber $X_s$ for
      $s\in S$, $\omega$ restricts to a non-degenerate exact
      $\iota_s^*\pi^*\mathcal{L}$-twisted closed 2-form $\omega_s :=
      \iota_s^*\omega$, which any trivialization of the line
      $\iota_s^*\pi^*\mathcal{L}$ identifies with an $n$-shifted
      \textit{exact} symplectic structure on $X_s$ (Lemma
      \ref{lem:fiberwise twist}).
  \end{itemize} We call $\lambda$ a \bfem{relative primitive}. For
  the globally trivial case,
  $\mathcal{L}\simeq\cO_S$, we drop $\mathcal{L}$ and speak simply of an
  \bfem{$n$-shifted exact symplectic fibration}.
\end{definition}

%As we will see in the proof of Theorem \ref{thm:2} below and by the
% following result from DSG literature, this notion of an $n$-shifted
% symplectic fibration is a convenient replacement for smooth ones.

\subsection{Compatibility and induced exact symplectic
fibrations}\label{sec:Ind_symp_fib}

We will now establish the first main result, Theorem \ref{thm:2}, of
this paper, stating that: \emph{Given a morphism $\pi: X \rightarrow
  S$ of derived Artin $\K$-stacks, if $X$ admits an $n$-shifted contact
  form $\alpha_X$ whose contact line bundle is pulled back from the
  base, $\cL_X \simeq \pi^*\cL$, and whose restriction
  $\iota_s^*d_{tot}\alpha_X$ to each geometric fiber $X_s$ is
  non-degenerate (equivalently, defines an $n$-shifted symplectic
  structure on $X_s$ after any trivialization of
  $\iota_s^*\pi^*\cL$), then $\pi$ admits a ``compatible'' $\cL$-twisted
$n$-shifted exact symplectic fibration structure.}

Before proving Theorem \ref{thm:2}, let us first introduce the
appropriate \bfem{compatibility} notion for shifted exact symplectic
fibrations in the DSG framework:% and then provide the proof of
% Theorem \ref{thm:2}:

\begin{definition}[Compatibility with a shifted exact symplectic
  fibration] \label{def:compatibility}
  Let $\pi: X \rightarrow S$ be a morphism of derived Artin
  $\K$-stacks such that $X$ admits an $n$-shifted contact form
  $\alpha_X \in \mathcal{A}^1(X, \cL_X, n)$ whose contact line bundle
  is pulled back from the base, $\cL_X \simeq \pi^*\cL$ for a line
  bundle $\cL$ on $S$. Denote by $\iota_s: X_s \hookrightarrow X$ the
  inclusion of the geometric fiber over $s\in S(\K)$. We say that
  $\alpha_X$ is \bfem{compatible with} $\pi$ if, for each $s\in
  S(\K)$, the twisted closed 2-form $\iota_s^*d_{tot}\alpha_X \in
  \mathcal{A}^{2,cl}(X_s, \iota_s^*\pi^*\cL, n)$ --- which is exact,
  with primitive $\iota_s^*\alpha_X$ --- is homotopic to a
  non-degenerate twisted 2-form; equivalently, by Lemma
  \ref{lem:fiberwise twist}, if any trivialization of the line
  $\iota_s^*\pi^*\cL$ identifies it, up to homotopy, with an
  $n$-shifted exact symplectic structure on $X_s$.

\end{definition}

\begin{proof}[\textbf{Proof of Theorem \ref{thm:2}}]
  We are given a morphism of derived Artin $\K$-stacks $\pi: X
  \rightarrow S$ and an absolute $n$-shifted contact structure
  $(X,\alpha_X, \cL_X[n]) \in \mathsf{Cont}(X, n)$ with $\cL_X \simeq
  \pi^*\cL$ for a line bundle $\cL$ on $S$, such that $\alpha_X$ is
  compatible with $\pi$.

  By the definition of relative forms \cite{PTVV}, the canonical
  projection $DR(X) \xrightarrow{\cdot/S} DR(X/S)$, tensored with
  $\pi^*\cL$, maps the twisted $n$-shifted 1-form $\alpha_X \in
  \mathcal{A}^{1}(X, \pi^*\cL, n)$ to a relative twisted $n$-shifted
  1-form $\alpha_{rel}:=\alpha_X/S \in
  \mathcal{A}^{1}(X/S, \pi^*\cL, n)$.
  Taking the total de Rham differential (Remark \ref{rmk:twisted
  d_dR}) yields an exact relative $\pi^*\cL$-twisted $n$-shifted
  closed 2-form
  \begin{equation}
    \omega_{rel}:= d_{tot} \alpha_{rel} \in \mathcal{A}^{2,cl}(X/S,
    \pi^*\cL, n),
  \end{equation}
  with relative primitive $\lambda := \alpha_{rel}$. On each
  geometric fiber, $\iota_s^*\omega_{rel} \simeq
  d_{tot}(\iota_s^*\alpha_{rel})$ is non-degenerate by compatibility,
  and is identified with an $n$-shifted exact symplectic structure by
  any trivialization of $\iota_s^*\pi^*\cL$ (Lemma
  \ref{lem:fiberwise twist}). 
  
  Since
  $\iota_s^*\mathbb{L}_{X/S} \simeq \mathbb{L}_{X_s}$ (Theorem
  \ref{thm:hag_cotangent_properties}-(\ref{item:cotangent_basechange})),
  the morphism $\omega_{rel}^\flat: \mathbb{T}_{X/S} \rightarrow
  \mathbb{L}_{X/S} \otimes \pi^*\cL[n]$ is an equivalence at every
  geometric point of $X$; by standard properties of perfect complexes
  (Theorem
  \ref{thm:HAG_perfect_properties}-(\ref{item:perfect_pointwise})),
  it is then an equivalence in $QCoh(X)$, exactly as in the proof of
  \cite[Theorem 1.1]{abi}. Hence $(\lambda, \omega_{rel})$ is an
  $\cL$-twisted $n$-shifted exact symplectic fibration structure on
  $\pi$ in the sense of Definition \ref{defn_shifted symp fibration},
  compatible with $\alpha_X$ by construction; the required variation
  of the proof of \cite[Theorem 1.1]{abi} consists precisely of
  carrying the twist by $\pi^*\cL$ through every step and identifying
  the fibers via Lemma \ref{lem:fiberwise twist}.
\end{proof}

%===================================
%===================================
%===================================

\section{Derived Contact  Thurston Theorem} \label{sec:results}

Our previous work \cite{abi} extends the concept of symplectic
fibration to the setting of derived symplectic geometry with the help
of relative derived structures, where the $n$-shifted relative
symplectic structures serve as an appropriate substitute for
classical symplectic fibrations.  We established two key theorems in
that paper:  The first discusses \textit{induced compatible shifted
symplectic fibrations}, and the second, called the
\textit{derived--symplectic--Thurston Theorem}, provides a method for
constructing a compatible (absolute) shifted symplectic structure on
the source $X$ of a shifted symplectic fibration \( \pi: X
\rightarrow (S,\omega_S) \),  with a shifted symplectic target. %This
% structure is in fact induced by the given relative shifted
% symplectic structure on the morphism \( \pi: X \rightarrow S \),
% where \( S \) is a derived stack with a shifted symplectic target.
% % leading to the \textit{derived version of Thurston's theorem}.

A contact analog of the (classical) symplectic Thurston Theorem
(which can be named as the \textit{contact Thurston Theorem}) has
been provided by the first author in \cite{mfa}, where the techniques
in proving symplectic Thurston Theorem are modified for the smooth
contact setting. More precisely, for a given exact symplectic
fibration $\pi:E\to B$ with compact fibers and a compact contact base
$B$, by summing a primitive of the exact symplectic structure on the
fibration $\pi$ with a suitable constant multiple of the pull-back of
the contact form on the base $B$, one can construct a contact form
(and so structure) on the total space $E$.

This section is devoted to establishing a \textit{contact version of
the derived symplectic Thurston Theorem} mentioned in the first
paragraph above. One should also note that the same result (see
below) can also be considered as a \textit{derived version of the
contact Thurston Theorem} discussed in the second paragraph above. In
what follows, we state and prove the precise version of Theorem
\ref{thm:1}, yielding a construction of absolute contact structures
over exact symplectic fibrations. Let us start with the main statement:

\begin{theorem}[Derived Contact Thurston Theorem for Artin Stacks]
  \label{thm: proof_thurston in dcg}
  Let $X$ and $S$ be derived Artin $\K$-stacks (both locally of
  finite presentation and quasi-compact) over a field $\K$ of
  characteristic zero. Suppose also that the base $S$ admits an
  $n$-shifted contact structure, with the twisting line bundle
  $\mathcal{L}_S$, and an $n$-contact form $\sigma \in
  \mathcal{A}^1(S, \cL_S, n)$. Let
  \[ \omega = d_{tot}\lambda \in \mathcal{A}^{2,cl}(X/S, \pi^*\cL_S, n) \]
  be a $\cL_S$-twisted $n$-shifted exact symplectic fibration
  structure on $\pi: X \rightarrow (S, \cL_S[n], \sigma)$, with
  relative primitive $\lambda \in \mathcal{A}^1(X/S, \pi^*\cL_S, n)$.

  Assume there is a global $\pi^*\cL_S$-twisted $n$-shifted 1-form
  $\Omega \in \mathcal{A}^{1}(X, \pi^*\cL_S, n)$ on $X$ with
  $\iota_s^*d_{tot}\Omega \sim \omega_s$ in
  $\mathcal{A}^{2,cl}(X_s,\iota_s^*\pi^*\cL_S, n)$, where $\iota_s:
  X_s \hookrightarrow X$ is the inclusion of the geometric fiber over
  $s\in S(\K)$ and $\omega_s\sim\iota_s^*\omega$.
  Then we have:
  \begin{enumerate}[\upshape(a)]
    \item (Lifting formalism) There exists an absolute $n$-shifted
      contact structure defined on a dense Zariski-open derived
      substack $U \subseteq X$, with contact line bundle $\pi^*\cL_S|_U$.

    \item (Homological construction) There exists an absolute
      $n$-shifted contact structure with a primitive form $\alpha_X$ such that
      \begin{enumerate}[\upshape(i)]
        \item $\alpha_X \sim \Omega + c \cdot \pi^* \sigma$ \, in
          $\mathcal{A}^{1}(U,\pi^*\cL_S,n)$ for a generic constant $c
          \in \K$, and
        \item $\iota_s^* d_{tot}\alpha_X \sim \omega_s$ \, in
          $\mathcal{A}^{2,cl}(U_s,\iota_s^*\pi^*\cL_S,n)$ for each
          $s\in S(\K)$, making $\alpha_X$ compatible with $\pi$.
      \end{enumerate}
      Furthermore, if $X$ is proper over $\K$, then $U = X$, and the
      structure extends globally.
    \item (Equivalence) The contact forms constructed in parts
      \textup{(a)} and \textup{(b)} agree.
  \end{enumerate}

\end{theorem}

\subsection{Proof of Theorem \ref{thm: proof_thurston in dcg}(a):
Existence via equivariant descent }\label{sec:contact_symp_fib}

%\subsection{Exact derived contact symplectic fibrations}

The classical Thurston construction was recently adapted to contact
geometry by the first author \cite{mfa}, demonstrating that the total
space of a symplectic fibration with exact symplectic fibers over a
contact base admits a compatible contact structure. We now aim to
establish the derived analogue of this result, providing a
construction for shifted contact structures on the total spaces of
exact derived symplectic fibrations over shifted contact bases. This
will be accomplished by synthesizing the (symplectic) derived
Thurston construction \cite[Theorem 4.1]{abi} with the derived
symplectification and transversality theorems given in
\cite{IzbudakBerktav2026}.

Let $\pi: X \to S$ be a derived fibration where the base $S$ is an
$n$-shifted contact derived Artin stack and the geometric fibers
$X_s$ are exact $n$-shifted symplectic after any trivialization of
the restricted line $\iota_s^*\pi^*\cL_S$ (Lemma
\ref{lem:fiberwise twist}). We construct an absolute
$n$-shifted contact structure on the total space $X$ through the
following sequence of steps.
\vspace{-0.05in}

\paragraph{Step 1: Derived symplectification of the base.}
Following \cite{IzbudakBerktav2026}, the $n$-shifted contact
structure on $S$ is governed by a contact line bundle
$\mathcal{L}_S$. The derived symplectification of $S$ is defined as
the total space of the associated principal $\mathbb{G}_m$-bundle
\begin{equation}
  p_S: \widetilde{S} \longrightarrow S.
\end{equation}
By the properties of derived symplectification, the derived stack
$\widetilde{S}$ carries an $n$-shifted symplectic structure
$\omega_{\widetilde{S}}$. This symplectic form operates as an
eigenform of weight 1 with respect to the $\mathbb{G}_m$-action on
the fibers of $p_S$. Moreover $\omega_{\widetilde{S}} =
d_{dR}\widetilde{\sigma}$, where $\widetilde{\sigma} \in
\mathcal{A}^1(\widetilde{S},n)^{\{1\}}$ is the weight $1$ form
corresponding to $\sigma$ under the weight dictionary of Notation
\ref{notn:twisted}.
\vspace{-0.05in}

\paragraph{Step 2: The derived (symplectic) Thurston construction on the lift.}
We pull back the fibration $\pi: X \to S$ along the torsor projection
$p_S$ to define a new fibration over the symplectified base. Let
$\widetilde{X}$ be the derived fiber product \(\widetilde{X} := X
\times_S \widetilde{S}\), with the homotopy commutative diagram
\[
  \begin{tikzcd}
    \tilde{X}:= X \times_S \tilde{S} \arrow[d] \arrow[r,
    "\tilde{\pi}"] & \tilde{S} \arrow[d, "p_S"] \\
    X \arrow[r, "\pi"]                                                  & S
\end{tikzcd}\]
which yields a pullback fibration $\tilde{\pi}: \widetilde{X} \to
\widetilde{S}$. We equip $\tilde{\pi}$ with the pulled-back fibration
structure $(\widetilde{\lambda}, \widetilde{\omega}) :=
(p_X^*\lambda, p_X^*\omega)$: the twist $p_X^*\pi^*\cL_S \simeq
\tilde{\pi}^*p_S^*\cL_S$ is canonically trivialized on
$\widetilde{X}$, so the weight dictionary (Notation
\ref{notn:twisted}) identifies $\widetilde{\lambda}$ and
$\widetilde{\omega}$ with untwisted relative forms of weight $1$;
moreover $\mathbb{L}_{\widetilde{X}/\widetilde{S}} \simeq
p_X^*\mathbb{L}_{X/S}$, so the relative non-degeneracy is preserved under
this base change. Hence $(\widetilde{\lambda}, \widetilde{\omega})$
is an $n$-shifted exact symplectic fibration structure on
$\tilde{\pi}$, whose fiber over a geometric point $\tilde{s} \in
\widetilde{S}$ above $s$ is $X_s$ equipped with the trivialization of
$\omega_s$ determined by the basis $\tilde{s}$ of $L_s$ (Lemma
\ref{lem:fiberwise twist}).

Under the same dictionary, $\Omega$ corresponds to a basic weight $1$
absolute $n$-shifted 1-form $\widetilde{\Omega}$ on $\widetilde{X}$,
and the hypothesis $\iota_s^*d_{tot}\Omega \sim \omega_s$ --- an
identity of $\iota_s^*\pi^*\cL_S$-twisted forms --- transfers to the
identity $\iota_{\tilde{s}}^*d_{tot}\widetilde{\Omega} \sim
\iota_{\tilde{s}}^*\widetilde{\omega}$ of ordinary forms at every
geometric point $\tilde{s}$ of $\widetilde{S}$, both sides being
trivialized by the same basis $\tilde{s}$. Thus the closed $2$-form
$d_{tot}\widetilde{\Omega}$ restricts on each geometric fiber of
$\tilde{\pi}$ to the exact symplectic structure of that fiber; this
is the hypothesis of \cite[Theorem 4.1]{abi} for $\tilde{\pi}$, which
therefore produces an
absolute $n$-shifted symplectic form
\begin{equation}
  \omega_{\widetilde{X}} := d_{dR}\widetilde{\Omega} + c \cdot
  \tilde{\pi}^*(\omega_{\widetilde{S}})
\end{equation}
on a dense Zariski-open derived substack $\widetilde{U} \subseteq
\widetilde{X}$, where $c \in \K$ is a generic scalar constant.
\vspace{-0.1in}

\paragraph{Step 3: Equivariant descent and contact reduction.}
The symplectic structure $\omega_{\widetilde{X}}$ on the lift must
descend to a contact structure on the total space $X$. This is
achieved via the derived transversality theorem of \cite{IzbudakBerktav2026}.

The space $\widetilde{X}$ inherits a $\mathbb{G}_m$-action from
$\widetilde{S}$, and $p_X: \widetilde{X} \rightarrow X$ is a principal
$\mathbb{G}_m$-bundle, being the pullback of $p_S$. By Theorem
\ref{thm_Symplectization}, $\omega_{\widetilde{S}}$ is an eigenform of
weight $1$, hence so is $\tilde{\pi}^*\omega_{\widetilde{S}}$;
moreover $d_{dR}\widetilde{\Omega}$ is of weight $1$ because
$\widetilde{\Omega}$ is and $d_{dR}$ is $\mathbb{G}_m$-equivariant.
Therefore, $\omega_{\widetilde{X}}$ is an eigenform of weight $1$.

The locus $\widetilde{U}$ may be taken to be maximal, namely the
vanishing locus of the cone of $\omega_{\widetilde{X}}^\flat:
\mathbb{T}_{\widetilde{X}} \to \mathbb{L}_{\widetilde{X}}[n]$: since
$\omega_{\widetilde{X}}$ has weight $1$, this morphism is
$\mathbb{G}_m$-equivariant up to the weight $1$ twist of its target,
so its cone is an equivariant perfect complex, its vanishing locus
$\widetilde{U}$ is $\mathbb{G}_m$-invariant, and
$\widetilde{U}=p_X^{-1}(U)$ for a dense Zariski-open derived substack
$U \subseteq X$. Consequently
$(\widetilde{U}, \omega_{\widetilde{X}}|_{\widetilde{U}})$ is an
$n$-shifted symplectic derived Artin stack equipped with a weight $1$
$\mathbb{G}_m$-action, and \cite[Theorem 3.7]{IzbudakBerktav2026}
endows the stack quotient
\begin{equation}
  [\widetilde{U}/\mathbb{G}_m] \simeq U
\end{equation}
with an $n$-shifted contact structure, whose contact line bundle is
$\pi^*\cL_S|_U$ and whose contact form is the descent of the weight
$1$ primitive $\widetilde{\Omega} + c \cdot
\tilde{\pi}^*\widetilde{\sigma}$. Indeed, the contact form of
\cite[Theorem 3.7]{IzbudakBerktav2026} (restated as Theorem
\ref{thm:derived transversality}) is the descent of the contraction
$\iota_Y\omega_{\widetilde{X}}$ of $\omega_{\widetilde{X}}$ with the
fundamental vector field $Y$ of the $\mathbb{G}_m$-action; since
$\widetilde{\Omega} + c \cdot \tilde{\pi}^*\widetilde{\sigma}$ is
horizontal of weight $1$, we have $\iota_Y(\widetilde{\Omega} + c
\cdot \tilde{\pi}^*\widetilde{\sigma}) \simeq 0$ while the Lie
derivative along $Y$ acts on it as the identity, so the Cartan
homotopy formula gives $\iota_Y\omega_{\widetilde{X}} \sim
\widetilde{\Omega} + c \cdot \tilde{\pi}^*\widetilde{\sigma}$.
\qed

\subsection{Proof of Theorem \ref{thm: proof_thurston in dcg}(b):
Primitives via homological construction}
The proof proceeds within the stable $\infty$-category of
quasi-coherent sheaves, ${QCoh}(X)$, and is based on the proof scheme
of \cite[Theorem 4.1]{abi}, with obvious modifications. We leave some
details to the reader, but provide an outline highlighting the
necessary modifications.   In line with the approach in \cite[Theorem
4.1]{abi}, the proof of Theorem \ref{thm: proof_thurston in dcg}(b)
is also divided into four steps.

\paragraph{Step 1: Definition of $\alpha_X$ and fiber compatibility.}
By assumption, we are given a global absolute $n$-shifted
$\pi^*\cL_S$-twisted 1-form $\Omega \in \mathcal{A}^{1}(X,
\pi^*\cL_S, n)$ whose de Rham differential restricts on any geometric
fiber $X_s$ to $\iota_s^*d_{dR}\Omega \sim \omega_s$ in
$\mathcal{A}^{2,cl}(X_s,\iota_s^*\pi^*\cL_S,n)$. We define the
candidate absolute $n$-shifted contact form on $X$ directly as:
\begin{equation} \label{eq:alpha_candidate}
  \alpha_X := \Omega + c \cdot \pi^*\sigma \quad \in
  \mathcal{A}^{1}(X, \pi^*\cL_S, n),
\end{equation}
where $c \in \K$ is a scalar constant (to be determined later) and
$\sigma$ is the $n$-contact form on $S$. Since $\mathcal{A}^1(X,
\pi^*\cL_S, n)$ is a
$\K$-module, the combination $\alpha_X$ is again a $\pi^*\cL_S$-twisted
$n$-shifted 1-form, and $d_{tot}\alpha_X = d_{dR}\Omega + c \cdot
\pi^*d_{dR}\sigma$.

Recall that for any $\K$-point $s \in S(\K)$, the geometric fiber
$X_s$ is defined by the following Cartesian diagram in the
$\infty$-category of derived stacks $dSt_{\K}$:
\begin{equation} \label{cd:fiber_pullback}
  \begin{tikzcd}[column sep=large, row sep=large]
    X_s \arrow[r, "\iota_s", hook] \arrow[d, "p_s"'] & X \arrow[d, "\pi"] \\
    \spec\, \K \arrow[r, "s"'] & S.
  \end{tikzcd}
\end{equation}
Since $\LL_{\spec \K} \simeq 0$ and $\pi \circ \iota_s \simeq s \circ
p_s$ factors through $\spec \K$, we have a canonical null-homotopy
$\iota_s^*\pi^*\sigma \simeq 0$, and evaluating the closed $2$-form
$d_{dR}\alpha_X$ on the fiber yields
\begin{equation} \label{eq:fiber_eval}
  \iota_s^*d_{dR}\alpha_X \simeq \iota_s^*d_{dR}\Omega + c \cdot
  \iota_s^*\pi^*d_{dR}\sigma \sim \omega_s + 0 \simeq \omega_s,
\end{equation}
which gives the desired compatibility condition for any choice of the
constant $c$.

\paragraph{Step 2: Homological splitting of the tangent complex.}
To prove that $\alpha_X$ defines a valid contact structure, we must
verify its non-degeneracy. Let $\omega_X := d_{tot}\alpha_X$ and
$\omega_\Omega := d_{tot}\Omega$. The 1-form $\alpha_X$ induces a
contraction morphism $\alpha_X^\vee: \mathbb{T}_X \to \pi^*\cL_S[n]$.
Then the derived contact distribution $\mathcal{K}_X \in QCoh(X)$ is
defined as the homotopy fiber of $\alpha_X^\vee$:
\begin{equation}
  \mathcal{K}_X \longrightarrow \mathbb{T}_X
  \xrightarrow{\alpha_X^\vee} \pi^*\cL_S[n].
\end{equation}
We must now prove the following claim.
\begin{claim}\label{claim_non-degenerate contact dist}
  The derived contact distribution is in fact non-degenerate, meaning
  that the induced morphism $\omega_X^\flat: \mathcal{K}_X \to
  \mathcal{K}_X^\vee \otimes \pi^*\cL_S[n]$ is an equivalence in
  $\mathrm{Ho}({QCoh}(X))$.
\end{claim}To establish Claim \ref{claim_non-degenerate contact
dist}, we first need some preliminary computations provided in the
rest of Step-2. The necessary follow-up constructions continue in
Step-3 and are finalized in Step-4. The proof of Claim
\ref{claim_non-degenerate contact dist} will then be completed at the
end of Step-4, which will also conclude the proof of Theorem
\ref{thm: proof_thurston in dcg}(b).

Let us begin with some preliminary observations. Recall that by
dualizing the canonical relative cotangent exact triangle (Theorem
\ref{thm:hag_cotangent_properties}-(\ref{item:cotangent_transitivity})),
the morphism $\pi: X \to S$ induces an exact triangle of perfect
tangent complexes in $QCoh(X)$:
\begin{equation} \label{eq:tangent_triangle}
  \begin{tikzcd}
    \mathbb{T}_{X/S} \ar[r, "i"] & \mathbb{T}_X \ar[r, "p"] &
    \pi^*\mathbb{T}_S \ar[r, "+1"] & \mathbb{T}_{X/S}[1].
  \end{tikzcd}
\end{equation}
We then define the relative projection of $\omega_\Omega^\flat$ onto
the vertical tangent complex to be the composition
$(\omega_\Omega)_{rel}^\flat := i^\vee{[n]} \circ \omega_\Omega^\flat
\circ i : \mathbb{T}_{X/S} \to \mathbb{L}_{X/S} \otimes
\pi^*\cL_S[n]$, factoring as follows:
\begin{equation} \label{cd:theta_projection}
  \begin{tikzcd}
    \mathbb{T}_{X/S} \arrow[r, "i"] \arrow[d, "\Theta_{rel}^\flat"']
    & \mathbb{T}_X \arrow[d, "\Theta^\flat"] \\
    \mathbb{L}_{X/S} \otimes \pi^*\cL_S[n] & \mathbb{L}_X \otimes
    \pi^*\cL_S[n] \arrow[l, "i^\vee{[n]}"'],
  \end{tikzcd}
\end{equation}
where $\Theta = \omega_\Omega$, and where, here and for the rest of
the proof, we abusively write $f^\vee[n]$ for $(f^\vee \otimes
\mathrm{id})[n]$ whenever the twist by $\pi^*\cL_S$ (or by $\cL_S$ on
$S$) is present in the target.
By hypothesis, $\iota_s^*\omega_\Omega \sim \omega_s$, a
non-degenerate exact $\iota_s^*\pi^*\cL_S$-twisted closed 2-form on
the fiber $X_s$: non-degeneracy holds because $\omega$ is a relative
symplectic structure, and is unaffected by the twist (Lemma
\ref{lem:fiberwise twist}). Because
the pullback of the relative cotangent complex to the geometric fiber
is the cotangent complex of the fiber ($\iota_s^*\mathbb{L}_{X/S}
\simeq \mathbb{L}_{X_s}$), the restriction of
$(\omega_\Omega)_{rel}^\flat$ to any geometric fiber $X_s$ over $s
\in S(\K)$ is a quasi-isomorphism.

By standard properties of perfect complexes (Theorem
\ref{thm:HAG_perfect_properties}-(\ref{item:perfect_pointwise})),
since $(\omega_\Omega)_{rel}^\flat$ is an equivalence at every
geometric point of $X$, it is a global equivalence in $QCoh(X)$.
Because $(\omega_\Omega)_{rel}^\flat$ is an equivalence, we construct
a canonical homological retraction morphism $\rho: \mathbb{T}_X \to
\mathbb{T}_{X/S}$ via $$\rho := ((\omega_\Omega)_{rel}^\flat)^{-1}
\circ i^\vee{[n]} \circ \omega_\Omega^\flat.$$ This splits the exact
triangle (\ref{eq:tangent_triangle}). Let $\eta: \pi^*\mathbb{T}_S
\to \mathbb{T}_X$ denote the canonical section of $p$ induced by this
splitting. Thus, in $\mathrm{Ho}({QCoh}(X))$, we obtain the canonical
equivalence:
\begin{align} \label{eq:splitting}
  \mathbb{T}_X &\simeq \mathbb{T}_{X/S} \oplus \pi^*\mathbb{T}_S.
\end{align}

\paragraph{Step 3: Block-Diagonalization of the Contact Distribution.}
Let $\omega_S := d_{tot}\sigma \in \mathcal{A}^{2,cl}(S,\cL_S,n)$. Since
$\sigma$ is an $n$-shifted contact form on $S$, it defines a contact
distribution $$\mathcal{K}_S = \mathrm{fib}(\sigma^\vee: \mathbb{T}_S
\to \cL_S[n]),$$ yielding an exact triangle $\mathcal{K}_S
\xrightarrow{j} \mathbb{T}_S \xrightarrow{\sigma^\vee}
\cL_S[n]$. Under the assumption that $S$ is contact,
$d_{tot}\sigma$ restricts to an equivalence
$\omega_{\mathcal{K}_S}^\flat : \mathcal{K}_S \xrightarrow{\sim}
\mathcal{K}_S^\vee \otimes \cL_S[n]$.
This equivalence provides a canonical homological left inverse to the
inclusion $j$, given by the composition $$r_S :=
(\omega_{\mathcal{K}_S}^\flat)^{-1} \circ j^\vee[n] \circ
\omega_S^\flat : \mathbb{T}_S \to \mathcal{K}_S,$$ which satisfies
$r_S \circ j \simeq \mathrm{id}_{\mathcal{K}_S}$ since $j^\vee[n]
\circ \omega_S^\flat \circ j \simeq \omega_{\mathcal{K}_S}^\flat$.
This canonically splits the exact triangle, yielding $\mathbb{T}_S
\simeq \mathcal{K}_S \oplus \cL_S[n]$ globally in $QCoh(S)$.
Pulling this back via $\pi$, we get a splitting of the horizontal
tangent complex $\pi^*\mathbb{T}_S \simeq \pi^*\mathcal{K}_S \oplus
\pi^*\cL_S[n]$. Combining this with \eqref{eq:splitting} yields:
\begin{equation}
  \mathbb{T}_X \simeq \mathbb{T}_{X/S} \oplus \pi^*\mathcal{K}_S
  \oplus \pi^*\cL_S[n].
\end{equation}
We evaluate the candidate contact 1-form $\alpha_X^\vee = \Omega^\vee
+ c \cdot \pi^*\sigma^\vee : \mathbb{T}_X \to \pi^*\cL_S[n]$ with
respect to this splitting. On the $\pi^*\cL_S[n]$ component coming
from the base, $\pi^*\sigma^\vee$ acts as an equivalence. Thus, the
restriction $\alpha_X^\vee|_{\pi^*\cL_S[n]} \simeq
\Omega^\vee|_{\pi^*\cL_S[n]} + c \cdot
\pi^*\sigma^\vee|_{\pi^*\cL_S[n]}$ can be factored as
$\pi^*\sigma^\vee \circ (c \cdot \mathrm{id} + K)$ for a global
endomorphism $K = (\pi^*\sigma^\vee)^{-1} \circ
\Omega^\vee|_{\pi^*\cL_S[n]}$ of the perfect complex $\pi^*\cL_S[n]$.
For a generic choice of scalar $c \in \K^\times$, $c$ avoids the
eigenvalues of $-K$ at generic points, so that $c \cdot \mathrm{id} +
K$ is invertible there, ensuring that $\alpha_X^\vee$ restricted to $\pi^*\cL_S[n]$
is an equivalence onto $\pi^*\cL_S[n]$ over a dense Zariski-open
derived substack $U_1 \subseteq X$.

Therefore, over $U_1$, write $E := \mathbb{T}_{X/S} \oplus
\pi^*\mathcal{K}_S$ and $F := \pi^*\cL_S[n]$, and set $a :=
\alpha_X^\vee|_E: E \to F$ and $b := \alpha_X^\vee|_F \simeq c \cdot
\mathrm{id}_F + K$, which is an equivalence over $U_1$. The morphism
\begin{equation} \label{eq:graph_section}
  \kappa := (\mathrm{id}_E, -b^{-1} \circ a) : E \longrightarrow E
  \oplus F \simeq \mathbb{T}_{U_1}
\end{equation}
satisfies $\alpha_X^\vee \circ \kappa \simeq a - b \circ b^{-1} \circ
a \simeq 0$, and comparing cones shows that it induces an equivalence
\begin{equation} \label{eq:contact_distribution_split}
  \kappa: \mathbb{T}_{X/S} \oplus \pi^*\mathcal{K}_S
  \xrightarrow{\ \sim\ } \mathcal{K}_{U_1}.
\end{equation}
Note that $\kappa$ is the graph of $-b^{-1}\circ a$ rather than the
evident inclusion; this is the only point at which the argument
departs from the symplectic case of \cite[Theorem 4.1]{abi}.

We evaluate $\omega_X^\flat|_{\mathcal{K}_{U_1}} := \kappa^\vee[n]
\circ \omega_X^\flat \circ \kappa$. Since $\pi^*\sigma^\vee \circ i
\simeq 0$ and $\pi^*\sigma^\vee \circ \eta \circ j \simeq 0$, the
morphism $a$ is independent of $c$; and since
$\pi^*(d_{tot}\sigma)^\flat$ factors through $p$, it vanishes on
$\mathbb{T}_{X/S}$ and contributes only the term $c \cdot
\pi^*(\omega_{\mathcal{K}_S})^\flat$ in the lower right corner.
Writing $\iota_E := (i, \eta \circ j)$ for the evident inclusion of
$E$ into $\mathbb{T}_{U_1}$ and $\Delta := \kappa - \iota_E$ for the
correction term of \eqref{eq:graph_section}, we obtain
\begin{equation} \label{eq:alpha_x_block}
  \omega_X^\flat|_{\mathcal{K}_{U_1}} \simeq
  \begin{pmatrix} (\omega_\Omega)_{rel}^\flat & 0 \\ 0 & C + c \cdot
    \pi^*(\omega_{\mathcal{K}_S})^\flat
  \end{pmatrix} + R(c),
\end{equation}
where $C: \pi^*\mathcal{K}_S \to \pi^*\mathcal{K}_S^\vee \otimes
\pi^*\cL_S[n]$
encapsulates the remaining evaluation of $\omega_\Omega$, the diagonal
part being computed exactly as in \cite[Theorem 4.1]{abi} using $\rho
\circ \eta \simeq 0$, and where
\[ R(c) := \iota_E^\vee[n] \circ \omega_X^\flat \circ \Delta +
  \Delta^\vee[n] \circ \omega_X^\flat \circ \iota_E + \Delta^\vee[n]
\circ \omega_X^\flat \circ \Delta \]
collects the terms involving $\Delta$. Each summand of $R(c)$ carries
at least one factor $b^{-1} \simeq (c \cdot \mathrm{id}_F +
K)^{-1}$, and the third carries two. Moreover, the first two summands
have no $c$-linear part: the morphism $\Delta$ factors through the
summand $\pi^*\cL_S[n]$, and the splitting $\mathbb{T}_S \simeq
\mathcal{K}_S \oplus \cL_S[n]$ is $\omega_S$-orthogonal by the
construction of $r_S$, so that $j^\vee[n] \circ \omega_S^\flat$
annihilates the summand $\cL_S[n]$ through which $\Delta$ factors.

\paragraph{Step 4: Generic invertibility and properness.}
Set $u := c^{-1}$ and rescale the second row, i.e. consider
\[ N(u) := \mathrm{diag}(\mathrm{id}, u \cdot \mathrm{id}) \circ
\omega_X^\flat|_{\mathcal{K}_{U_1}} \]
as a family of morphisms of perfect complexes over
$\mathbb{A}^1_{u}$. By the description of $R(c)$ in Step 3 every
entry of $N(u)$ is regular at $u=0$ and the contribution of $R(c)$
vanishes there, so that
\[ N(0) \simeq
  \begin{pmatrix} (\omega_\Omega)_{rel}^\flat & 0 \\ 0 &
    \pi^*(\omega_{\mathcal{K}_S})^\flat
\end{pmatrix}. \]
The first block is an equivalence by Step 2 and the second because
$(S,\sigma)$ is contact, so $N(0)$ is an equivalence. Since the locus
on which a perfect complex is acyclic is open, at each generic point
$\xi_j$ of $X$ the morphism $N(u) \otimes \kappa(\xi_j)$ fails to be
a quasi-isomorphism only for $u$ in a finite subset $F_j \subset
\mathbb{A}^1_{\kappa(\xi_j)}$ with $0 \notin F_j$. It remains to
control these finitely many bad values.
Factoring out $\pi^*(\omega_{\mathcal{K}_S})^\flat$, which is an
equivalence, the non-degeneracy condition on the horizontal block
becomes the invertibility of:
\begin{equation} \label{eq:horizontal_block}
  C + c \cdot \pi^*(\omega_{\mathcal{K}_S})^\flat \simeq
  \pi^*(\omega_{\mathcal{K}_S})^\flat \circ \left( c \cdot
  \mathrm{id}_{\pi^*\mathcal{K}_S} + M \right),
\end{equation}
where $M = (\pi^*(\omega_{\mathcal{K}_S})^\flat)^{-1} \circ C$ is a
global endomorphism of the perfect complex $\pi^*\mathcal{K}_S$.

Since $X$ is a quasi-compact Artin stack over a field of
characteristic zero, its underlying topological space is Noetherian
and has finitely many irreducible components. Let $\xi_1, \dots,
\xi_m$ denote the generic points of these components. By the exact
spectral argument established in the proof of symplectic Thurston
theorem in \cite[Theorem 4.1]{abi}, the eigenvalues of $K$ and $M$ at
these generic points form strictly finite sets $E_j', E_j \subset
\overline{\kappa(\xi_j)}$.
We thus select the constant $c \in \K^\times$ with the property that
$-c \notin \bigcup_{j=1}^m ((E_j' \cup E_j) \cap \K)$ and $c^{-1}
\notin \bigcup_{j=1}^m (F_j \cap \K)$, which is possible since $\K$
is infinite and the excluded set is finite.

With this choice, $c \cdot \mathrm{id} + K$ and $c \cdot \mathrm{id}
+ M$ are equivalences at all generic points $\xi_j$. The vanishing
locus of their cones forms a dense Zariski-open derived Artin
substack $U \subseteq {U_1}$. On $U$, the horizontal block is an
equivalence, making $\omega_X$ non-degenerate on $\mathcal{K}_U$, and
hence establishing Claim \ref{claim_non-degenerate contact dist}. In
total, $\alpha_X$ thus defines \textit{an absolute $n$-shifted
contact structure on $U$.}
\medskip

Furthermore, if $X$ is proper over $\K$, then the perfect complexes
of endomorphisms evaluating generic eigenvalues are global
finite-dimensional $\K$-algebras. A generic choice of $c$ ensures
strict invertibility globally everywhere, so $U = X$ and the contact
structure extends globally.

We then complete the proof of Theorem \ref{thm: proof_thurston in dcg}(b).

\qed

\subsection{Proof of Theorem \ref{thm: proof_thurston in dcg}(c):
Equivalence of the two constructions}
Suppose $\pi\colon X \to S$ is an $n$-shifted exact symplectic
fibration over an (absolute) $n$-shifted contact base $(S, \sigma,
\cL_S[n])$. We demonstrate in this section that the
\textit{homological construction} (Theorem \ref{thm: proof_thurston
in dcg}(b)) of the derived contact Thurston theorem coincides with
the formalism of \textit{equivariant descent via derived
symplectification}  (Theorem \ref{thm: proof_thurston in dcg}(a)),
eliminating potential discrepancies between the two approaches.

Let $p_S\colon \widetilde{S} \to S$ be the derived symplectification
of the base, structured as a principal $\mathbb{G}_m$-bundle
associated with $\mathcal{L}_S$. Define $\widetilde{X} := X \times_S
\widetilde{S}$ to be the pullback along $\pi$, making $p_X\colon
\widetilde{X} \to X$ a principal $\mathbb{G}_m$-bundle, with the diagram
\[
  \begin{tikzcd}
    \tilde{X}:= X \times_S \tilde{S} \arrow[d, "p_X"] \arrow[r,
    "\tilde{\pi}"] & \tilde{S} \arrow[d, "p_S"] \\
    X \arrow[r, "\pi"]                                                  & S.
\end{tikzcd}\]
By the weight dictionary of Notation \ref{notn:twisted}, the twisted
$n$-shifted contact form $\sigma \in \mathcal{A}^{1}(S,
\mathcal{L}_S, n)$ on the base $S$ canonically corresponds to a
weight 1 absolute $n$-shifted 1-form
$\widetilde{\sigma} \in \mathcal{A}^{1}(\widetilde{S}, n)^{\{1\}}$ on
$\widetilde{S}$. The derived symplectification of $S$ carries the
exact $n$-shifted symplectic structure $\omega_{\widetilde{S}} =
d_{dR}\widetilde{\sigma}$.

Under the equivalence between twisted 1-forms on the base $X$
(sections of $\LL_X \otimes \pi^*\mathcal{L}_S[n]$) and weight 1
equivariant horizontal 1-forms on the total space of the principal
$\mathbb{G}_m$-bundle $p_X\colon \widetilde{X} \to X$, the global
compatible twisted 1-form $\Omega$ canonically lifts to a horizontal
1-form $\widetilde{\Omega} \in \mathcal{A}^{1}(\widetilde{X}, n)^{\{1\}}$.

Applying the exact symplectic Thurston theorem \cite[Theorem
4.1]{abi} to the pullback fibration $\tilde{\pi}\colon \widetilde{X}
\to \widetilde{S}$, equipped with the pulled-back fibration structure
as in Section \ref{sec:contact_symp_fib}, the
absolute exact symplectic form constructed on the total space
$\widetilde{X}$ is defined as
\begin{equation}
  \omega_{\widetilde{X}} := d_{dR}\widetilde{\Omega} + c \cdot
  \tilde{\pi}^*\omega_{\widetilde{S}},
\end{equation}
for a generic scalar $c \in \K^\times$.
Substituting the relation $\omega_{\widetilde{S}} = d_{dR}
\widetilde{\sigma}$ and utilizing the commutativity of the derived
pullback $\tilde{\pi}^*$ with the derived de Rham differential
$d_{dR}$, we obtain
\begin{equation}
  \omega_{\widetilde{X}} = d_{dR}\widetilde{\Omega} + c \cdot
  \tilde{\pi}^* (d_{dR} \widetilde{\sigma}) =
  d_{dR}\left(\widetilde{\Omega} + c \cdot \tilde{\pi}^*
  \widetilde{\sigma}\right).
\end{equation}

We set the primitive 1-form on $\widetilde{X}$ as $\widetilde{\alpha}
:= \widetilde{\Omega} + c \cdot \tilde{\pi}^* \widetilde{\sigma}$.
Because $\widetilde{\Omega}$ and $\widetilde{\sigma}$ are both
equivariant eigenforms of weight 1 under the $\Gm$-action, their
linear combination $\widetilde{\alpha}$ evaluates as a weight 1 form,
yielding $\widetilde{\alpha} \in \mathcal{A}^{1}(\widetilde{X},
n)^{\{1\}}$. Therefore, $\omega_{\widetilde{X}} = d_{dR}
\widetilde{\alpha}$ is an exact $n$-shifted symplectic structure
governed by a weight 1 primitive.

The derived transversality theorem (cf. Theorem \ref{thm:derived
transversality}) asserts that the stack quotient of the weight 1
exact symplectic locus $(\widetilde{U},
d_{dR}\widetilde{\alpha}|_{\widetilde{U}})$ --- where $\widetilde{U}
= p_X^{-1}(U)$ is the dense Zariski-open $\mathbb{G}_m$-invariant
non-degeneracy locus of Section \ref{sec:contact_symp_fib} --- by
the $\mathbb{G}_m$-action produces an $n$-shifted contact structure
on $U \subseteq X$. The associated contact form on $U$ is exactly
the $\mathbb{G}_m$-equivariant descent of the weight 1 primitive
$\widetilde{\alpha}$.

Because equivariant descent establishes an equivalence of stable
$\infty$-categories between weight 1 horizontal forms on
$\widetilde{X}$ and twisted forms on $X$, it acts linearly and
preserves additive structures. The descent of the horizontal weight 1
form $\widetilde{\Omega}$ recovers $\Omega$, and the descent of
$\tilde{\pi}^*\widetilde{\sigma} = p_X^* \pi^* \sigma$ recovers
$\pi^*\sigma$. Thus, the geometric descent operation yields precisely
the twisted 1-form
\begin{equation}
  \alpha_X = \Omega + c \cdot \pi^*\sigma \quad \in \pi_0
  \mathbb{R}\Gamma(X, \LL_X \otimes^{\mathbb{L}}_{\cO_X} \pi^*\cL_S[n]).
\end{equation}
\begin{remark}[Equivalence of Constructions]
  \label{rmk:equivalence_of_constructions}
  Let us summarize what we have proved above and give the upshot: The
  two approaches for constructing an absolute shifted contact
  structure on the total space $X$ of an exact symplectic fibration
  over a contact base---\textit{the homological construction and the
    $\mathbb{G}_m$-equivariant derived symplectification
  formalism}---produce identical contact forms up to canonical
  equivalence in the space of twisted 1-forms, removing the necessity
  of an intermediate ``contactomorphism''.  We refer to the companion
  paper \cite{IzbudakBerktav4} for the formal theory of derived
  contactomorphism.
\end{remark}

\begin{remark}[Different Choices of Primitives]
  \label{rmk:contactomorphism_different_primitives}
  Building on the contactomorphism framework established in the
  companion paper \cite{IzbudakBerktav4}, we can address alternative
  choices of the initial compatible 1-form.

  Let $\Omega, \Omega' \in \pi_0 \mathbb{R}\Gamma(X, \LL_X
  \otimes^{\mathbb{L}}_{\cO_X} \pi^*\cL_S[n])$ be two global twisted
  $n$-shifted 1-forms on the total space $X$ of an exact symplectic
  fibration $\pi\colon X \to S$. Suppose both satisfy the fiberwise
  compatibility $\iota_s^*d_{tot}\Omega \sim \omega_s \sim
  \iota_s^*d_{tot}\Omega'$ for each $s \in S(\K)$, and that they differ globally by an internal $d$-exact and
  a de Rham exact discrepancy: $\Omega' - \Omega = d\gamma +
  d_{dR}\psi$, for some 1-form $\gamma$ of internal degree $n-1$ and
  a $d$-closed 0-form $\psi$ of degree $n$.

  The resulting absolute contact forms obtained via the derived
  Thurston method evaluate to $\alpha_X = \Omega + c \cdot
  \pi^*\sigma$ and $\alpha'_X = \Omega' + c \cdot \pi^*\sigma$. Their
  difference evaluates as
  \begin{equation}
    \alpha_X' - \alpha_X = \Omega' - \Omega = d\gamma + d_{dR}\psi.
  \end{equation}

  This relation, in fact, establishes that the identity map
  $\mathrm{id}_X$ constitutes a ``weak contact equivalence'' between
  the two structures in the sense of \cite{IzbudakBerktav4}. As shown
  in  \cite{IzbudakBerktav4}, while it cannot be deformed to a strict
  equivalence algebraically in the global setting, applying the
  formal deformation procedure of \cite{IzbudakBerktav4} canonically
  adjusts the identity map via the formal flow of the derived Reeb
  vector field over the formal power series ring. This yields a
  formal strict $n$-shifted contactomorphism $\widehat{\varphi}\colon
  X \times \Spec \K[[t]] \xrightarrow{\sim} X \times \Spec \K[[t]]$
  interpolating perfectly between the formal extensions of the
  derived contact structures $(X, \pi^*\cL_S[n], \alpha_X)$ and $(X,
  \pi^*\cL_S[n], \alpha'_X)$. This clarifies the formal geometric
  consequence of modifying local and global primitives.
\end{remark}

We then complete the proof of Theorem \ref{thm: proof_thurston in
dcg}(c),   and hence that of Theorem \ref{thm:1}.

\qed

%== =================================
%===================================
%===================================
We close the section with the most obvious (basic) example:
\begin{example}
  If $(X,\lambda_X)$ is an $n$-shifted exact symplectic $\K$-stack
  and $(Y, \cL_Y, \sigma_Y)$ is an $n$-shifted contact $\K$-stack
  with symplectification $p_Y: \widetilde{Y} \rightarrow Y$, then
  $X\times Y$ admits a \bfem{product} $n$-shifted contact form with
  contact line bundle $\pi_Y^*\cL_Y$, where $\pi_X: X\times Y \to X$
  and $\pi_Y:X\times Y \to Y$ are the projection maps. Indeed,
  letting $t$ denote the tautological weight $1$ trivialization of
  $p_Y^*\cL_Y$ on $\widetilde{Y}$, the basic weight $1$ $1$-form $t
  \cdot \pi_X^*\lambda_X$ on $X \times \widetilde{Y}$ descends, by
  the weight dictionary (Notation \ref{notn:twisted}), to a twisted
  1-form $\Omega_{X\times Y} \in \mathcal{A}^1(X\times Y,
  \pi_Y^*\cL_Y, n)$ restricting on each fiber of $\pi_Y$ to a twist
  of $\lambda_X$ (Lemma \ref{lem:fiberwise twist}). Then
  $\alpha_{X\times Y}=\Omega_{X\times Y} + \pi^*_Y \sigma_Y$ is a
  product $n$-shifted contact form on $X \times Y$, and one can
  easily check that $\pi_Y$ is a \bfem{trivial} $\cL_Y$-twisted
  $n$-shifted exact symplectic fibration compatible with
  $\alpha_{X\times Y}$.
\end{example}

\section{Examples} \label{sec:applications}
\subsection{Conormal stacks}
The following results apply the derived contact Thurston theorem to
construct absolute shifted contact structures on the source stacks of
exact symplectic fibrations.% given in \cite[Examples 3.9-3.11]{abi}

Throughout this subsection, $(S, \cL_S, \sigma)$ is an $(n-1)$-shifted
contact derived Artin $\K$-stack and $p_S: \widetilde{S} \rightarrow
S$ denotes its symplectification, i.e. the principal $\mathbb{G}_m$-bundle
of $\cL_S$. For a morphism $f: X \rightarrow S$, we write $\tilde{f}
:= f \times_S \widetilde{S}: \widetilde{X} \rightarrow \widetilde{S}$
for its base change, where $\widetilde{X}=X \times_S \widetilde{S}$. Since the contact line bundle $\cL_S$ is in
general non-trivial and we do not assume coorientability a priori, the
conormal construction must itself be twisted by $\cL_S$; the correct
object is the following quotient, in parallel with the weight
dictionary of Notation \ref{notn:twisted}.

\begin{definition}[Twisted conormal stacks] \label{defn:twisted conormal}
  Let $f: X \rightarrow S$ be a morphism of derived Artin $\K$-stacks
  locally of finite presentation. By functoriality of the $n$-shifted
  conormal stack and the base change $N^*[n]\tilde{f} \simeq N^*[n]f
  \times_S \widetilde{S}$ (Theorem
  \ref{thm:hag_cotangent_properties}-(\ref{item:cotangent_basechange})),
  the $\mathbb{G}_m$-action on $\widetilde{S}$ induces an action on
  $N^*[n]\tilde{f}$; we let $\mathbb{G}_m$ act on $N^*[n]\tilde{f}$
  by the \emph{diagonal} of this action and the weight $1$ scaling of
  the conormal directions. The \bfem{$\cL_S$-twisted $n$-shifted
  conormal stack} of $f$ is the quotient
  \[ N^*_{\cL_S}[n]f := [\, N^*[n]\tilde{f} \,/\, \mathbb{G}_m \,], \]
  a derived stack over $[\widetilde{X}/\mathbb{G}_m] \simeq X$, and
  hence over $S$; write $\pi: N^*_{\cL_S}[n]f \rightarrow S$ for the
  projection and $p: N^*[n]\tilde{f} \rightarrow N^*_{\cL_S}[n]f$ for
  the quotient map. The pair of $p$ and the structure morphism to
  $\widetilde{S}$ identifies $N^*[n]\tilde{f} \simeq N^*_{\cL_S}[n]f
  \times_S \widetilde{S}$, so $p$ is a principal
  $\mathbb{G}_m$-bundle, canonically the pullback of $p_S$ along
  $\pi$, and its associated weight $1$ line bundle is $\pi^*\cL_S$.
  Over a geometric point $s \in S(\K)$, any basis $\ell_s$ of $L_s :=
  s^*\cL_S$ identifies the fiber $\pi^{-1}(s)$ with the shifted
  cotangent stack $T^*[n-1]X_s$ of the fiber $X_s$ (as
  $\mathbb{L}_{\widetilde{X}/\widetilde{S}}$ restricts to
  $\mathbb{L}_{X_s}$), two such identifications differing by the
  weight $1$ scaling (cf. Lemma \ref{lem:fiberwise twist}).
\end{definition}

\begin{proposition}[Conormal Stacks]\label{prop: conormal}
  Let $f: X \rightarrow S$ be a morphism of derived Artin $\K$-stacks
  locally of finite presentation, with $N^*_{\cL_S}[n]f$ and $S$
  quasi-compact, such that $S$ admits an $(n-1)$-shifted contact
  structure with contact line bundle $\cL_S$ and $(n-1)$-contact form
  $\sigma \in \mathcal{A}^1(S, \cL_S, n-1)$. 
  We then have:
  \begin{enumerate}[(i)]
      \item  The morphism $\pi:
  N^*_{\cL_S}[n]f \rightarrow S$ carries a canonical
  $\cL_S$-twisted $(n-1)$-shifted exact symplectic fibration
  structure $\omega = d_{tot}\lambda$ whose fiberwise restriction
  $\omega_s := \iota_s^*\omega$ is identified, by any basis $\ell_s$
  of $L_s$, with the canonical $(n-1)$-shifted exact symplectic
  structure of $T^*[n-1]X_s$ (Lemma \ref{lem:fiberwise twist}).
  \item If, in addition, there exists a $\pi^*\cL_S$-twisted
  $(n-1)$-shifted 1-form $\Omega \in \mathcal{A}^{1}(N^*_{\cL_S}[n]f,
  \pi^*\cL_S, n-1)$ with $\iota_s^*d_{tot}\Omega \sim \omega_s$ in
  $\mathcal{A}^{2,cl}(\pi^{-1}(s), \iota_s^*\pi^*\cL_S, n-1)$ for
  every $s \in S(\K)$.
  Then there exists an $(n-1)$-shifted contact structure, with
  contact line bundle the restriction of $\pi^*\cL_S$, on a dense
  Zariski-open derived substack of the $\cL_S$-twisted $n$-shifted
  conormal $\K$-stack $N^*_{\cL_S}[n]f$, compatible with the twisted
  exact symplectic fibration $\pi: N^*_{\cL_S}[n]f \to S$.
  \end{enumerate}

\end{proposition}

\begin{proof}
  \textbf{Step 1: The untwisted construction over the
  symplectification.} Consider the base change $\tilde{f}:
  \widetilde{X} \rightarrow \widetilde{S}$. The canonical fiber
  sequence $\tilde{f}^*\mathbb{L}_{\widetilde{S}} \rightarrow
  \mathbb{L}_{\widetilde{X}} \rightarrow
  \mathbb{L}_{\widetilde{X}/\widetilde{S}}$ induces an $n$-shifted
  Lagrangian structure on the morphism $\tilde{g}: N^*[n]\tilde{f}
  \longrightarrow T^*[n]\widetilde{S}$, where the target is equipped
  with its canonical $n$-shifted exact symplectic structure
  $d_{dR}\lambda_{\widetilde{S}}$ ($\lambda_{\widetilde{S}}$ being
  the tautological Liouville 1-form \cite[\S 2.1]{Calaque2019}).
  These fit into the commutative diagram:

  \begin{equation}
    \begin{tikzcd}
      N^*[n]\tilde{f} \arrow[r, "\tilde{g}"] \arrow[rd,
      "{\tilde{\pi}}"', dashed] & T^*[n]\widetilde{S}
      \arrow[d, "\pi_{\widetilde{S}}"] \\
      & \widetilde{S}
    \end{tikzcd}
  \end{equation}
  where $\pi_{\widetilde{S}}$ is an $n$-shifted Lagrangian fibration
  \cite[Theorem 2.4]{Calaque2019}. By \cite[Proposition
  1.10]{Safronov2023}, the composition $\tilde{\pi} =
  \pi_{\widetilde{S}} \circ \tilde{g}$ therefore carries an
  $(n-1)$-shifted symplectic fibration structure $\widetilde{\omega}
  \in \mathcal{A}^{2,cl}(N^*[n]\tilde{f}/\widetilde{S}, n-1)$. Its
  fiberwise values are constructed as in the proof of \cite[Theorem
  2.22]{Safronov2019}: for each $\tilde{s} \in \widetilde{S}$,
  consider
  \[
    \begin{tikzcd}
      {
      {N^*[n]\tilde{f} \times_{T^*[n]\widetilde{S}}
      (T^*[n]\widetilde{S})_{\tilde{s}}}} \arrow[rr, "h"]
      \arrow[d, "j_{\tilde{s}}", hook] &  &
      {(T^*[n]\widetilde{S})_{\tilde{s}}} \arrow[r,
      "\pi_{\widetilde{S}}|"]
      \arrow[d, "i_{\tilde{s}}", hook] & \{\tilde{s}\} \arrow[d, "i",
      hook] \\
      {N^*[n]\tilde{f}} \arrow[rr, "\tilde{g}"] \arrow[rrr,
      "\tilde{\pi}=\pi_{\widetilde{S}} \circ \tilde{g}",
      bend right]              &  & {T^*[n]\widetilde{S}} \arrow[r,
      "\pi_{\widetilde{S}}"]
      & \widetilde{S}
    \end{tikzcd}
  \]
  where $i_{\tilde{s}}$ and $j_{\tilde{s}}$ denote the inclusions of
  the fibers of $\pi_{\widetilde{S}}$ and $\tilde{\pi}$ over
  $\tilde{s}$, and $h$ the induced map. The commutativity of the
  left rectangle and the Lagrangian structures on $\tilde{g}$ and
  $i_{\tilde{s}}$ yield a composition of homotopies $$0 \sim
  j_{\tilde{s}}^*\tilde{g}^*d_{dR}\lambda_{\widetilde{S}} \sim
  h^*i_{\tilde{s}}^*d_{dR}\lambda_{\widetilde{S}} \sim 0$$ in
  $\mathcal{A}^{2,cl}(\tilde{\pi}^{-1}(\tilde{s}), n)$; this loop at
  the zero form defines a class in
  $\pi_1(\mathcal{A}^{2,cl}(\tilde{\pi}^{-1}(\tilde{s}), n), 0)$,
  which, by the Dold-Kan correspondence and the loop-space
  equivalence $\Omega\mathcal{A}^{2,cl}(-,n) \simeq
  \mathcal{A}^{2,cl}(-,n-1)$, canonically evaluates to
  $\widetilde{\omega}_{\tilde{s}} \in \pi_0
  \mathcal{A}^{2,cl}(\tilde{\pi}^{-1}(\tilde{s}), n-1)$, which is
  non-degenerate by \cite[Theorem 2.9]{PTVV}. The structure is
  exact: the tautological form vanishes canonically on the conormal,
  i.e. there is a canonical null-homotopy of
  $\tilde{g}^*\lambda_{\widetilde{S}}$, and applying the same
  loop-space argument to $\lambda_{\widetilde{S}}$ in place of
  $d_{dR}\lambda_{\widetilde{S}}$ yields a relative primitive
  $\widetilde{\lambda} \in
  \mathcal{A}^1(N^*[n]\tilde{f}/\widetilde{S}, n-1)$ with
  $d_{dR}\widetilde{\lambda} = \widetilde{\omega}$. In particular,
  each geometric fiber $\tilde{\pi}^{-1}(\tilde{s}) \simeq
  T^*[n-1]X_s$ carries its canonical $(n-1)$-shifted exact
  symplectic structure.
\medskip

  \textbf{Step 2: Equivariance.} Equip $N^*[n]\tilde{f}$ and
  $T^*[n]\widetilde{S}$ with the diagonal $\mathbb{G}_m$-actions of
  Definition \ref{defn:twisted conormal}. By naturality, the
  tautological form $\lambda_{\widetilde{S}}$ is invariant under the
  functorial lift of the base action and has weight $1$ under the
  scaling of the cotangent directions; hence
  $\lambda_{\widetilde{S}}$, and with it,
  $d_{dR}\lambda_{\widetilde{S}}$ and
  $\tilde{g}^*\lambda_{\widetilde{S}}$ are eigenforms of weight $1$
  for the diagonal action. Every ingredient of Step 1 --- $\tilde{g}$,
  $\pi_{\widetilde{S}}$, the Lagrangian structures induced by the
  canonical fiber sequence, the null-homotopy of
  $\tilde{g}^*\lambda_{\widetilde{S}}$, and the loop-space
  construction --- is canonical, hence $\mathbb{G}_m$-equivariant.
  Consequently $\widetilde{\lambda}$ and $\widetilde{\omega}$ are
  eigenforms of weight $1$.
\medskip

  \textbf{Step 3: Descent.} By Definition \ref{defn:twisted
  conormal}, $p: N^*[n]\tilde{f} \rightarrow N^*_{\cL_S}[n]f$ is a
  principal $\mathbb{G}_m$-bundle, the pullback of $p_S$ along $\pi$,
  with associated weight $1$ line bundle $\pi^*\cL_S$; moreover
  $\mathbb{L}_{N^*[n]\tilde{f}/\widetilde{S}} \simeq
  p^*\mathbb{L}_{N^*_{\cL_S}[n]f/S}$ (Theorem
  \ref{thm:hag_cotangent_properties}-(\ref{item:cotangent_basechange})).
  The weight dictionary of Notation \ref{notn:twisted}, applied to
  relative forms through this identification, therefore converts the
  weight $1$ pair $(\widetilde{\lambda}, \widetilde{\omega})$ into a
  $\pi^*\cL_S$-twisted pair
  \[ \lambda \in \mathcal{A}^1(N^*_{\cL_S}[n]f/S, \pi^*\cL_S, n-1),
  \qquad \omega := d_{tot}\lambda \in
  \mathcal{A}^{2,cl}(N^*_{\cL_S}[n]f/S, \pi^*\cL_S, n-1). \]
  Relative non-degeneracy may be checked after the smooth surjective
  base change $p$, where it is the non-degeneracy of
  $\widetilde{\omega}$ established in Step 1; hence $(\lambda,
  \omega)$ is a $\cL_S$-twisted $(n-1)$-shifted exact symplectic
  fibration structure on $\pi$ in the sense of Definition
  \ref{defn_shifted symp fibration}. Restricting to a geometric
  fiber, the trivialization of $\iota_s^*\pi^*\cL_S$ by a basis
  $\ell_s = \tilde{s}$ recovers $\widetilde{\omega}_{\tilde{s}}$,
  the canonical structure of $T^*[n-1]X_s$; this proves the first
  assertion $(i)$.
\medskip 

  \textbf{Step 4: The contact Thurston theorem.} The pair $(\lambda,
  \omega)$ and the assumed twisted 1-form $\Omega$ with
  $\iota_s^*d_{tot}\Omega \sim \omega_s$ constitute precisely the
  hypotheses of Theorem \ref{thm: proof_thurston in dcg} for the morphisms $\pi:
  N^*_{\cL_S}[n]f \rightarrow (S, \cL_S[n], \sigma)$. By Theorem
  \ref{thm: proof_thurston in dcg}, with a suitable choice of a
  constant $c \in \K$, the twisted 1-form $\alpha_{N^*_{\cL_S}[n]f}
  := \Omega + c \cdot \pi^*\sigma$ satisfies
  $\alpha_{N^*_{\cL_S}[n]f} \in \mathsf{Cont}(U, n-1)$ for a dense
  open substack $U \subseteq N^*_{\cL_S}[n]f$, with contact line
  bundle $\pi^*\cL_S|_U$, making it a compatible $(n-1)$-shifted
  contact structure. If $N^*_{\cL_S}[n]f$ is proper, this extends
  globally. This establishes the second
  assertion $(ii)$ and concludes the proof.

\end{proof}

Note that the above result extends, by the same descent, to the
$\beta$-deformed conormal stacks $N_{\beta}^*[n] f \to
T_{\beta}^*[n]S$ of \cite{abi} for each $n$, provided the
deformation parameter is itself twisted. Let us sketch the proof: for $\beta \in
\mathcal{A}^{1,cl}(S, \cL_S, n+1)$, the weight $1$ avatar
$\widetilde{\beta} \in \mathcal{A}^{1,cl}(\widetilde{S},
n+1)^{\{1\}}$ of $\beta$ makes the graph
$\Gamma_{\widetilde{\beta}}: \widetilde{S} \rightarrow
T^*[n+1]\widetilde{S}$ \ $\mathbb{G}_m$-equivariant for the diagonal
action, so the construction of Steps 1--3 above applies to
$N^*_{\widetilde{\beta}}[n]\tilde{f}$ verbatim and descends to the
quotient $N^*_{\cL_S,\beta}[n]f :=
[N^*_{\widetilde{\beta}}[n]\tilde{f}/\mathbb{G}_m]$. One thus
readily establishes:

\begin{corollary}[Deformed conormal stacks] \label{cor: conormal}
  Let $f: X \rightarrow S$ be as in Prop. \ref{prop:
  conormal}, $\beta \in \mathcal{A}^{1,cl}(S, \cL_S, n+1)$, and
  suppose $N^*_{\cL_S,\beta}[n]f$ is quasi-compact. Then $\pi:
  N^*_{\cL_S,\beta}[n]f \rightarrow S$ carries a canonical
  $\cL_S$-twisted $(n-1)$-shifted exact symplectic fibration
  structure $\omega$. If, in addition, there exists a
  $\pi^*\cL_S$-twisted $(n-1)$-shifted 1-form $\Omega \in
  \mathcal{A}^{1}(N^*_{\cL_S,\beta}[n]f, \pi^*\cL_S, n-1)$ with
  $\iota_s^*d_{tot}\Omega \sim \omega_s := \iota_s^*\omega$ for
  every $s \in S(\K)$, then there exists an $(n-1)$-shifted contact
  structure, with the contact line bundle as the restriction of
  $\pi^*\cL_S$, on a dense Zariski-open derived substack of
  $N^*_{\cL_S,\beta}[n]f$, compatible with the fibration map $\pi$. \qed
\end{corollary}

\subsection{Contact AKSZ formalism and quotient mapping stacks}

Let us first start with a \textit{non-example} in which the derived symplectic
Thurston theorem is applicable, but its contact counterpart, Theorem
\ref{thm: proof_thurston in dcg}, is not. This is actually another
situation in which the current contact setting and arguments diverge
from the symplectic case of \cite[Theorem 4.1]{abi}.

\begin{remark}[Moment map quotients] \label{rmk: moment quotient}
  The quotient of an exact shifted symplectic stack by a moment map
  is not an application of Theorem \ref{thm: proof_thurston in dcg},
  and we record here what the symplectic Thurston theorem does give.
  Let $G$ be reductive with a non-degenerate invariant pairing on
  $\mathfrak{g}$, so that $BG$ carries a $2$-shifted symplectic
  structure $\omega_{BG}$, let $(X,\lambda_X)$ be an exact
  $2$-shifted symplectic derived Artin $\K$-stack with a smooth
  $G$-action, and let $\mu: X \rightarrow \mathfrak{g}^*[2]$ be a
  $2$-shifted moment map, with $X/G$ and $BG$ quasi-compact. By
  \cite[Example 3.11]{abi} the composition
  $$\Pi: X/G \longrightarrow \mathfrak{g}^*[2]/G \simeq T^*[3]BG
  \longrightarrow (BG, \omega_{BG})$$ carries a $2$-shifted exact
  symplectic fibration structure, and since $\mu$ is a moment map
  there is an equivalence of $2$-shifted exact symplectic stacks
  $X/G \times_{T^*[3]BG} \mathfrak{g}^*[2] \simeq X$ thanks to
  \cite{GrataloupPHD}, so that the exact structure on each fiber of
  $\Pi$ is $d_{tot}\lambda_X$. If there is a global absolute
  $2$-shifted 1-form $\Omega$ on $X/G$ whose de Rham differential
  restricts on each fiber to $d_{tot}\lambda_X$, then \cite[Theorem
  4.1]{abi} produces an absolute $2$-shifted \emph{symplectic}
  structure $\omega_{X/G} := d_{dR}\Omega + c \cdot
  \Pi^*\omega_{BG}$ on a dense Zariski-open derived substack $U
  \subseteq X/G$, for a generic $c \in \K$.

  The base $BG$ here is symplectic and not contact, so Theorem
  \ref{thm: proof_thurston in dcg} does not apply and no contact
  structure on $X/G$ arises from this route. If $X/G$ carries in
  addition a $\mathbb{G}_m$-action for which $U$ may be chosen
  $\mathbb{G}_m$-invariant and $\omega_{X/G}$ is an eigenform of
  weight $1$ on $U$, then \cite[Theorem
  3.7]{IzbudakBerktav2026} endows $[U/\mathbb{G}_m]$ with a
  $2$-shifted contact structure; note that the shift is preserved by
  this descent, and is not lowered to $1$.
\end{remark}

As a genuine example of our framework, the \textit{quotient mapping stack construction}
described in \cite[Example 3.10]{abi} can be naturally upgraded to an
application of the contact derived Thurston theorem.

We use the following notion of ``Legendrian fibrations'', to be
treated systematically in the forthcoming companion paper
\cite{abiRel}: In brief, by an
\bfem{$n$-shifted Legendrian fibration structure} on a morphism $p:
Y \rightarrow S$ of derived Artin $\K$-stacks, we mean a twisted $1$-form data $(\alpha_Y, \cL_Y[n])$ of an $n$-shifted
contact structure on $Y$ (non-degeneracy \bfem{not} assumed), together with
a trivialization of the relative form $\alpha_{Y/S}$ such that the
induced sequence $\mathbb{T}_{Y/S} \rightarrow \mathcal{K}_Y
\rightarrow \mathbb{L}_{Y/S}[n]$ is a stable fiber sequence. The
total space of a Legendrian fibration is then automatically
genuinely $n$-shifted contact, and the fiber inclusions are
$n$-shifted Legendrian \cite{abiRel}; the
composition principle for a Legendrian morphism followed by a
Legendrian fibration --- a \textit{relative version of the Legendrian
intersection theorem of \cite{IzbudakBerktav2026}} --- will appear
there as well. Below we construct its twisted exact symplectic
avatar on quotient mapping stacks, as required by Theorem \ref{thm:
proof_thurston in dcg}. Since the notion of
\cite{abiRel} places no constraint on the base, the
interaction with the contact structure on $S$ enters through one
additional hypothesis: the twisting line bundle of $Y$ is pulled
back, $\cL_Y \simeq p^*\cL_S$.

\begin{proposition}[Quotient Mapping Stacks] \label{prop:
  mapping_stack_thurston}
  Let $Y$ be an $n$-shifted contact derived Artin $\K$-stack and $X$
  be a smooth projective Calabi-Yau $m$-fold. Suppose $f: L
  \rightarrow Y$ is an $n$-shifted Legendrian morphism and $p: Y
  \rightarrow S$ is an $n$-shifted Legendrian fibration in the sense
  above (cf. \cite{abiRel}), such that the base $S$ admits an
  $(n-1)$-shifted contact structure $\sigma_S$ with the contact line
  bundle $\cL_S$, and the twisting line bundle of $Y$ is pulled back
  from the base, $\cL_Y \simeq p^*\cL_S$.

  Let $\widetilde{Y}, \widetilde{S}$ denote the \emph{derived
  symplectifications} of $Y, S$, and set $\widetilde{L} := L \times_Y
  \widetilde{Y}$, the principal $\mathbb{G}_m$-bundle of $f^*\cL_Y$;
  by the hypothesis $\cL_Y \simeq p^*\cL_S$, we have $\widetilde{Y}
  \simeq Y \times_S \widetilde{S}$ and $\widetilde{L} \simeq L
  \times_S \widetilde{S}$. The data fits into the diagram
  \[\begin{tikzcd}
\widetilde{L} \arrow[rr, "\tilde{f}"] \arrow[rrrd] \arrow[d] &  & \widetilde{Y} \arrow[rd, "\tilde{p}"] \arrow[d] &                         \\
L \arrow[rr, "f"] \arrow[rrrd, "p\circ f"']                  &  & Y \arrow[rd, "p"]                               & \widetilde{S} \arrow[d] \\
    &  &                                                 & S.                     
\end{tikzcd}\]
  
  We define
  the \bfem{quotient mapping stacks} $L_X := [\mathsf{Map}(X,
  \widetilde{L})/\mathbb{G}_m]$, \,$Y_X := [\mathsf{Map}(X,
  \widetilde{Y})/\mathbb{G}_m]$, and $S_X := [\mathsf{Map}(X,
  \widetilde{S})/\mathbb{G}_m]$.

  Assume additionally that the quotient mapping stacks $L_X$ and
  $S_X$ are quasi-compact (e.g., by restricting to appropriate open
  substacks or connected components), that the quotient mapping stack
  $S_X$ admits an $(n-m-1)$-shifted contact structure $\sigma_{S_X}$
  whose contact line bundle $\cL_{S_X}$ is the weight $1$ line bundle
  associated with the principal $\mathbb{G}_m$-bundle
  $\mathsf{Map}(X, \widetilde{S}) \rightarrow S_X$ (e.g., the
  structure furnished by \cite[Theorem A]{IzbudakBerktav3} whenever
  it applies), and that the induced $(n-1)$-shifted exact symplectic
  fibration $\tilde{p} \circ \tilde{f}: \widetilde{L} \rightarrow
  \widetilde{S}$ admits a compatible absolute $(n-1)$-shifted 1-form
  $\Omega_{\widetilde{L}}$ of weight 1 on $\widetilde{L}$.
  \medskip

  Then there is an $(n-m-1)$-shifted contact structure, with contact
  line bundle being the restriction of $\Pi^*\cL_{S_X}$, on a dense
  Zariski-open derived substack of the quotient mapping stack $L_X$
  which is compatible with the induced $\cL_{S_X}$-twisted exact
  symplectic fibration $\Pi: L_X \to S_X$.
\end{proposition}
\begin{proof}
  By the symplectic--contact dictionary (cf.
  \cite{IzbudakBerktav2026}), the Legendrian structure on $f$ and the
  Legendrian fibration structure on $p$, together with the
  identification $\cL_Y \simeq p^*\cL_S$, lift to a
  $\mathbb{G}_m$-equivariant $n$-shifted Lagrangian structure on
  $\tilde{f}: \widetilde{L} \rightarrow \widetilde{Y}$ and a
  $\mathbb{G}_m$-equivariant $n$-shifted Lagrangian fibration
  structure on $\tilde{p}: \widetilde{Y} \rightarrow \widetilde{S}$.
  Applying the AKSZ transgression, as established by the derived
  contact AKSZ formalism \cite{IzbudakBerktav3} (see also Section
  \ref{sec:terminology}), the morphism $\mathsf{Map}(X,
  \widetilde{L}) \rightarrow \mathsf{Map}(X, \widetilde{Y})$ carries
  an equivariant $(n-m)$-shifted Lagrangian structure and
  $\mathsf{Map}(X, \widetilde{Y}) \rightarrow \mathsf{Map}(X,
  \widetilde{S})$ an equivariant $(n-m)$-shifted Lagrangian fibration
  structure, all transgressed AKSZ data being
  eigenforms of weight $1$, since transgression along the fundamental
  class $[X]$ is linear in the weight $1$ symplectic data of the
  symplectifications. By \cite[Proposition 1.10]{Safronov2023},
  applied upstairs, the composition $\widetilde{\Pi}: \mathsf{Map}(X,
  \widetilde{L}) \rightarrow \mathsf{Map}(X, \widetilde{S})$
  therefore carries an $(n-m-1)$-shifted symplectic fibration
  structure, exact with relative primitive the transgressed Liouville
  form; by canonicity, this exact pair is an eigenform of weight $1$.
  Taking the stack quotient by $\mathbb{G}_m$ yields the diagram
  \begin{equation*}
    \begin{tikzcd}
      L_X \arrow[r, "f_X"] \arrow[rd, "\Pi"', dashed] & Y_X \arrow[d, "p_X"] \\
      & S_X.
    \end{tikzcd}
  \end{equation*}

  Since $\widetilde{L} \simeq L \times_S \widetilde{S}$ by the
  hypothesis $\cL_Y \simeq p^*\cL_S$, and $\mathsf{Map}(X,-)$
  preserves this fiber product, the quotient map $\mathsf{Map}(X, \widetilde{L})
  \rightarrow L_X$ is a principal $\mathbb{G}_m$-bundle, canonically
  the pullback along $\Pi$ of $\mathsf{Map}(X, \widetilde{S})
  \rightarrow S_X$, with associated weight $1$ line bundle
  $\Pi^*\cL_{S_X}$. Exactly as in Steps 2--3 of the proof of
  Proposition \ref{prop: conormal}, the weight dictionary (Notation
  \ref{notn:twisted}) therefore converts the weight $1$ exact pair
  above into an $\cL_{S_X}$-twisted $(n-m-1)$-shifted exact
  symplectic fibration structure on $\Pi: L_X \rightarrow S_X$.

  Because we assume that the induced fibration $\tilde{p} \circ
  \tilde{f}: \widetilde{L} \rightarrow \widetilde{S}$ admits a
  compatible absolute $(n-1)$-shifted 1-form $\Omega_{\widetilde{L}}$
  of weight 1, the AKSZ transgression functional allows one to simply
  set $$\Omega_{\mathsf{Map}} := \int_{[X]} \mathrm{ev}^*
  \Omega_{\widetilde{L}},$$ an $(n-m-1)$-shifted 1-form of weight $1$
  on $\mathsf{Map}(X, \widetilde{L})$, which descends, again by the
  weight dictionary, to a twisted 1-form $\Omega \in
  \mathcal{A}^{1}(L_X, \Pi^*\cL_{S_X}, n-m-1)$ with
  $\iota_s^*d_{tot}\Omega \sim \omega_s := \iota_s^*\omega$ for every
  $s \in S_X(\K)$, any trivialization of the line
  $\iota_s^*\Pi^*\cL_{S_X}$ identifying $\omega_s$ with the canonical
  $(n-m-1)$-shifted exact symplectic structure of the fiber (Lemma
  \ref{lem:fiberwise twist}).

  As assumed, the target quotient mapping stack $S_X$ is equipped
  with the $(n-m-1)$-shifted contact structure $\sigma_{S_X}$, with
  contact line bundle $\cL_{S_X}$.

  In total, we end up with the situation where the map $$\Pi: L_X
  \longrightarrow (S_X, \cL_{S_X}[n-m-1], \sigma_{S_X})$$ is an
  $\cL_{S_X}$-twisted $(n-m-1)$-shifted exact symplectic fibration
  over an $(n-m-1)$-shifted contact base, along with a global
  $\Pi^*\cL_{S_X}$-twisted $(n-m-1)$-shifted 1-form $\Omega$ on $L_X$
  satisfying the fiberwise compatibility condition. This matches the
  setup of the derived contact Thurston theorem (Theorem \ref{thm:
  proof_thurston in dcg}), provided we restrict to quasi-compact open
  substacks.

  Theorem \ref{thm: proof_thurston in dcg} then applies, and we
  conclude that for a generic scalar $c \in \K$, the twisted 1-form
  $$\alpha_{L_X} := \Omega + c \cdot \Pi^*\sigma_{S_X}$$ defines an
  absolute $(n-m-1)$-shifted contact structure, with contact line
  bundle the restriction of $\Pi^*\cL_{S_X}$, on a dense
  Zariski-open derived substack of $L_X$.
\end{proof}

\subsection{Affine exact symplectic fibrations}\label{sec:affine
symplectic fibration}
In this section, we focus on the affine case. In brief, we present an
affine model for an $n$-shifted exact symplectic fibration
$\spec\beta: \spec A \longrightarrow (\spec B, \sigma_B, \cL_{\spec
B}[n])$  over an $n$-shifted contact base, induced from a submersion
$\beta:B\rightarrow A$ of standard form cdgas. 

It should be noted that $(\spec B,
\sigma_B)$ will be taken in contact Darboux form (Theorem \ref{contact
darboux}), so the contact line bundle will  be trivialized, $\cL_{\spec B}
\simeq \cO_{\spec B}$; accordingly, all forms in this subsection are thus
untwisted, and Theorem \ref{thm: proof_thurston in dcg} is applied
with respect to this trivialization. Furthermore, by
choosing a natural 1-form on $\spec A$ that satisfies the hypothesis
of Theorem \ref{thm: proof_thurston in dcg}, we will
explicitly describe
the induced shifted contact structure on the source $\spec A$.

\paragraph{Step-1: Choosing $(\spec B, \sigma_B)$ in contact Darboux
form.} Assume that $n<0$ is an odd integer; say, $n=-2\ell-1$ for $\ell \in \N$.
Suppose $(B, \sigma_B)$ is in $n$-shifted contact Darboux form.
Briefly, we assume that $B$ is a (minimal) standard form cdga given
as a free algebra over a smooth $\K$-algebra $B(0)$ with coordinates
$x_1^0, \dots, x_{m_0}^0$, generated by the variables $\tilde{z}^n \in B^n$, and
$x^{-i}_j, y^{n+i}_j \in B$, where
\begin{align}
  & x_1^{-i}, \dots, x_{m_i}^{-i}& &\text{ in degree } -i
  \ \ \ \text{ for } i= 1,  \dots, \ell, \label{var set1_base} \\
  & y_1^{n+i}, \dots, y_{m_i}^{n+i}& & \text{ in degree } n+i
  \ \text{ for } i=0,\dots, \ell, \label{var set2_base}
\end{align}
where $m_1,\dots, m_{\ell}$ are non-negative integers. The
\emph{internal differential} $d_B$ on $B$ is determined by the
Hamiltonian $H\in B^{n+1}$ such that $d_B\tilde{z}^n \text{ is proportional to } H + d[\cdots]$ similar to (\ref{defn_internal d contact}).
Furthermore, $\sigma_B:=(\sigma^0, 0,  \dots)$, with
$  \sigma^0= d_{dR}\tilde{z}^n + \sum_{i,j} y_j^{n+i}
d_{dR}x_j^{-i},$ is an $n$-shifted contact form on $\spec B$ in
Darboux form. Since $d_B \sigma^0 = 0$, its de Rham differential is
$$d_{tot}\sigma^0 = d_{dR}\sigma^0 = \sum_{i,j} d_{dR}y_j^{n+i} d_{dR}x_j^{-i}.$$

\paragraph{Step-2: Setting the source $\spec A$.} Let $A^0:=A(0)$ be
a smooth algebra of dimension $m_0+n_0$, and let
$\beta^0:B^0\rightarrow A^0$ be a smooth morphism. Assume there exist
elements $u^0_1, \dots, u^0_{n_0}$ in $A^0$ such that
$\{d_{dR}\tx_{1}^0, \dots, d_{dR}\tx^0_{m_0}, d_{dR}u^0_1, \dots,
d_{dR}u^0_{n_0} \}$ is a basis over $A^0$ for $\Omega_{A^0}^1$, with
$\tx^0_j=\beta^0(x^0_j)$.

Letting $s=\ell$, we define the \emph{commutative graded algebra} $A$
to be the free graded algebra over $A(0)$ generated by the variables:
\begin{align}
  & \tx_1^{-i}, \dots, \tx_{m_i}^{-i} & & \text{ in degree } (-i)
  \text{ for } i= 1, \dots, \ell, \nonumber \\
  & \ty_1^{n+i}, \dots, \ty_{m_i}^{n+i} & & \text{ in degree } (n+i)
  \text{ for } i= 0, \dots, \ell, \nonumber \\
  & \hat{z}^n & & \text{ in degree } n, \nonumber \\
  & u_1^{-i}, \dots, u_{n_i}^{-i} & & \text{ in degree } (-i) \text{
  for } i=1, \dots, s, \nonumber \\
  & v_1^{n+i}, \dots, v_{n_i}^{n+i} & & \text{ in degree } (n+i)
  \text{ for } i=0,\dots, s.
\end{align}

Next, choose an element $G \in A^{n+1}$ (the \emph{superpotential})
depending only on the variables $u$ and $v$, satisfying the classical
master equation. We define the \emph{differential $d_A$ on $A$} via:
\begin{align*}
  &d_A|_{A^0}=0, \quad d_A\tx_j^{-i}= \beta(d_B x_j^{-i}), \quad
  d_A\ty_j^{n+i}= \beta(d_B y_j^{n+i}), \quad d_A\hat{z}^n =
  \beta(d_B \tilde{z}^n), \\
  & d_Au_j^{-i}= \partial G/ \partial v_j^{n+i}, \quad d_Av_j^{n+i}=
  \pm \partial G/ \partial u_j^{-i}.
\end{align*}
Here we extended $\beta^0$ to a graded algebra morphism
$\beta:B\rightarrow A$ by setting $\beta(x_j^{-i}) = \tx_j^{-i}$,
$\beta(y_j^{n+i}) = \ty_j^{n+i}$, and $\beta(\tilde{z}^n) = \hat{z}^n$.

\paragraph{Step-3: Setting the map from $\spec A$ to $\spec B$.}
The commutation relation $d_A \circ \beta = \beta \circ d_B$ holds
consistently on all generators by design, confirming that
$\beta:B\rightarrow A$ is a valid morphism of \emph{cdgas}.
Furthermore, $\beta$ acts as a submersion, which yields the desired
morphism $$\spec\beta: \spec A \longrightarrow (\spec B, \sigma_B)$$
of derived affine schemes over an $n$-shifted contact base.

\paragraph{Step-4: Defining a relative $n$-shifted exact symplectic
structure on $\spec \beta.$} Note that in the relative setting, $\dR
\tx_j^{-i}=0$, $\dR \ty_j^{n+i}=0$, and $\dR \hat{z}^n=0$ in
$\Omega^1_{A/B}$ because they belong to the image of $\beta.$ We
construct a relative exact 1-form $\lambda_{rel} = \sum_{i,j}
v_j^{n+i} \dR u_j^{-i}$. Let $\gamma $ be the relative $n$-shifted
2-form on $\spec \beta$ defined by \[ \gamma= d_{dR}\lambda_{rel} =
  \sum_{i,j} \dR v_j^{n+i} \dR u_j^{-i} \in \displaystyle
\bigwedge\nolimits^2 \Omega^1_{A/B}[n]. \]
Clearly, $d_{dR} \gamma=0$. Furthermore, exactly as in the symplectic
case, the classical master equation for $G$ implies $d_A\gamma = 0$.
Since $\gamma^{\flat}$ acts isomorphically on the vertical tangent
vectors, $\gamma$ defines a non-degenerate relative $n$-shifted exact
symplectic structure on $\spec \beta$.

\paragraph{Step-5: Choosing a suitable $\Omega \in
\mathcal{A}^{1}(\spec A, n).$} Recall the null-homotopic sequence of complexes:
\[DR(B) \xrightarrow[]{(\spec \beta)^*} DR(A) \xrightarrow[]{\cdot/B}
DR(A/B).\] We wish to define an absolute $n$-shifted 1-form
$\Omega:=(\Omega^0, 0,\dots)$ on $\spec A$ such that the image of
$d_{dR}\Omega$ under the map $\cdot /B$ is homotopic to $\gamma.$
Using the variables above, we set:
\[ \Omega^0 := \sum_{i,j} v_j^{n+i} \dR u_j^{-i} + \sum_{i,j}
\ty_j^{n+i} \dR \tx_j^{-i} + \dR\hat{z}^n. \]
Notice that $d_{dR}\Omega^0 - \gamma = \sum_{i,j} \dR \ty_j^{n+i} \dR
\tx_j^{-i}$, and we compute:
\begin{align*}
  d_{dR}\Omega^0 - \gamma &= \beta_* \bigg(\sum_{i,j} \dR y^{n+i}_j
  \dR x^{-i}_j\bigg) = (\spec \beta)^* (d_{dR}\sigma^0).
\end{align*}
Since the composite $DR(B) \rightarrow DR(A) \rightarrow DR(A/B)$ is
null-homotopic, the difference $d_{dR}\Omega^0 - \gamma$ maps to zero
in $DR(A/B)$, so $d_{dR}\Omega$ satisfies the fiberwise compatibility
condition.

\paragraph{Step-6: Constructing an $n$-shifted contact structure on
$\spec A.$} Using the element $\Omega$ introduced above, we simply
define the candidate absolute 1-form \[ \alpha_A:= \Omega^0 + c \cdot
(\spec \beta)^* (\sigma^0) \quad \text{ for a generic scalar } c\in \K.\]

Let $\omega_A = d_{tot}\alpha_A = d_{tot}\Omega^0 + c \cdot
(\spec\beta)^*(d_{tot}\sigma^0)$. To verify that $\alpha_A$ defines a
valid contact structure, we must evaluate the non-degeneracy
condition on the contact distribution $\mathcal{K}_A$. By
construction, there are no cross-terms in $\omega_A$ between the
relative vertical variables (the $u$ and $v$ generators) and the
horizontal variables (the $\tx$, $\ty$, and $\hat{z}$ generators).
Consequently, the contact distribution evaluation splits explicitly
as a block-diagonal matrix over the decomposition $$\mathcal{K}_A
\simeq (\spec \beta)^* (\mathcal{K}_B) \oplus \mathbb{T}_{A/B}.$$
Specifically, its block-diagonal components are exactly
$\gamma^\flat$ and $(1+c) \cdot (\spec
\beta)^*(d_{dR}\sigma^0)^\flat$. Since $\gamma$ is a relative
$n$-shifted exact symplectic structure (making $\gamma^\flat$ an
equivalence vertically) and $\sigma_B$ is an $n$-shifted contact
structure on the base (making $d_{dR}\sigma^0$ an equivalence
horizontally on $\mathcal{K}_B$), their induced block morphisms are
independently global equivalences. Therefore, the combined block
matrix is globally invertible on $\mathcal{K}_A$ for any
constant $c \in \K\setminus\{-1\}$. This confirms that
$d_{tot}\alpha_A$ is non-degenerate on the contact distribution,
thereby establishing $\alpha_A$ as an absolute $n$-shifted contact
structure on $\spec A$.\qed

%===================================
%===================================
%===================================
\section{Concluding remarks}

The prequel work \cite{abi} extended the concept of symplectic
fibration to the setting of derived symplectic geometry with the help
of relative derived structures, where the $n$-shifted relative
symplectic structures serve as a substitute for classical symplectic
fibrations. \cite{abi} established two key theorems: The first gives
rise to the derived analogue of the \textit{induced compatible
symplectic fibration lemma}, while the second, called the
\textit{derived (symplectic) Thurston theorem}, provides a method for
constructing a compatible (absolute) shifted symplectic structure on
the source $X$ of a symplectic fibration \( \pi: X \rightarrow
(S,\omega_S) \),  with a shifted symplectic target.

In this paper, on the other hand, we extend a prior work of Arıkan
\cite{mfa} --the classical contact Thurston framework-- to the
context of derived contact geometry in the spirit of our prior work
\cite{abi} on the derived symplectic Thurston theorem. Our
formulation relies on the natural symplectic–contact dictionary,
combined with the approach of \cite{abi}.
%Inspired by a prior work of Arıkan \cite{mfa}
In that respect, we consider exact symplectic fibrations, and it has
been shown in this paper that our relative-to-absolute construction
formalism is naturally applied to the case of \textit{contact
symplectic fibrations}, in which one considers exact symplectic
fibrations with a contact base. %and extends the contact Thurston
% theorem of \cite{mfa} to the derived setting as well.
%It has been shown in \cite{mfa} that Thurston's theorem can be
% adapted to the framework of such fibrations, allowing the
% construction of a \textit{compatible} contact structure on the source space.

In brief, as achieved in Section \ref{sec:contact_symp_fib}, this
paper successfully extends the main results of \cite{mfa} to the
setting of {derived contact geometry}, completely addressing the
question of establishing a derived contact analogue of Thurston's construction.

As an application, we combine, under certain conditions, Theorem
\ref{thm: proof_thurston in dcg} with particular examples of exact
symplectic fibrations to produce relative-to-absolute constructions
(cf. Theorem \ref{cor: app to Thm1}). We in fact explicitly resolved
the problem of this relative-to-absolute mechanism %constructing
% exact derived contact symplectic fibrations (cf. Section
% \ref{sec:contact_symp_fib}) %and $G$-equivariant shifted symplectic
% fibrations (cf. Section \ref{sec:g_equiv_fib})
by integrating our framework with the formalisms of derived
symplectification and transversality presented in
\cite{IzbudakBerktav2026}. Finally, Section \ref{sec:affine
symplectic fibration} examines affine shifted exact symplectic
fibrations with contact base and constructs a natural affine contact
model in accordance with  Theorem \ref{thm: proof_thurston in dcg}.

\paragraph{Outlook: $G$-equivariant derived fibrations.} Our program of fully
extending the theory of fibrations to the derived context still
requires further investigation. In classical (symplectic-contact)
theory, fibrations are naturally assumed to be locally trivializable
such that the corresponding structure group for the transition maps
is also part of the fibration data. For the case of symplectic
fibrations, in particular, all transition maps are in fact
symplectomorphisms of the fiber; whence, the structure group $G$ happens to be a subset of the symplectic group.  To capture this
structure-group data and the local triviality property, the sequel
will aim to develop \textit{$G$-equivariant derived fibrations} incorporating the theory of $G$-equivariant derived
symplectic geometry.
%===================================
%===================================
%===================================
\appendix
\section{Background from derived algebraic geometry}\label{appendix_DAG}
\subsection{Derived schemes and Artin stacks}

In derived geometry, there  exists an appropriate concept of a
\emph{spectrum functor} described as the right adjoint to the global
algebra of functions functor
\begin{equation*}
  \Gamma: dSt_{\K} \leftrightarrows (cdga_{\K}^{\leq 0})^{\mathrm{op}}: \spec,
\end{equation*} where $dSt_{\K}$ denotes the \bfem{$\infty$-category
of derived stacks}, with objects being \emph{$\infty$-sheaves on the
site $(dAff)^{\mathrm{op}}:= cdga_{\mathbb{K}}^{\leq 0}$}. We
formally denote elements in this opposite category $
(cdga_{\mathbb{K}}^{\leq 0})^{\mathrm{op}} $ by $\spec A$. A  derived
stack $X$ is in fact an $\infty$-groupoid-valued homotopy sheaf. More
precisely, we have:

\begin{definition}
  A \bfem{derived stack} $ {X} $ is  a  functor  $$ {X}:
  cdga_{\K}^{\leq 0} \rightarrow  Grpd_{\infty}, \ \ A\mapsto X(A)
  \simeq Map_{dStk_{\mathbb{K}}}(\spec A, {X}),$$ satisfying a
  descent condition.
  For more details, we refer to \cite{ToenHAG}.
\end{definition}

Let us now focus on particular, more tractable derived stacks with
good properties.

\begin{definition}
  An object ${X}$ in $dSt_{\mathbb{K}}$ is called an \bfem{affine
  derived $\mathbb{K}$-scheme} if  $X\simeq \spec A $ for some cdga
  $A \in cdga_{\K}^{\leq 0}$. An object $ X$ in $dSt_{\mathbb{K}}$ is
  then called a \bfem{derived $\mathbb{K}$-scheme} if it can be
  covered by Zariski open affine derived $\mathbb{K}$-schemes.
\end{definition}
Note that any affine derived $\mathbb{K}$-scheme $X$ can be
corepresented by a cdga $A$ such that $\spec A$ is viewed as the
functor \[ \spec A: B  \longmapsto Hom_{cdga_{\mathbb{K}}^{\leq 0}} (A,B).\]
Denote by $dSch_{\K} \subset dSt_{\mathbb{K}}$ the full
\bfem{$\infty$-subcategory of  derived $\mathbb{K}$-schemes}, and we
simply write $dAff_{\K} \subset dSch_{\mathbb{K}}$ for the full
\bfem{$\infty$-subcategory of  affine derived $\mathbb{K}$-schemes.}

There is another useful type of derived stacks, which generalizes
derived schemes but still with reasonable and tractable properties:
the concept of \emph{derived Artin $\K$-stack}.
From the foundational framework of To\"en and Vezzosi \cite{ToenHAG},
the concept of \textit{geometricity} for derived stacks is needed and
defined inductively relative to a chosen class of morphisms:

\begin{definition}[{\cite[Definition 1.3.3.1]{ToenHAG}}] \label{def:n_geometric}
  \begin{enumerate}
    \item A derived stack $X$ is \bfem{$(-1)$-geometric} if it is
      representable (i.e., an affine derived $\K$-scheme).
    \item A morphism of derived stacks $f: X \rightarrow Y$ is
      \bfem{$(-1)$-representable} if for any affine derived
      $\K$-scheme $U$ and any morphism $U \rightarrow Y$, the
      homotopy pullback $X \times^h_Y U$ is an affine derived $\K$-scheme.
  \end{enumerate}
  For $m \ge 0$, we define inductively:
  \begin{enumerate}
      \setcounter{enumi}{2}
    \item A derived stack $X$ is \bfem{$m$-geometric} if the diagonal
      morphism $X \rightarrow X \times^h X$ is $(m-1)$-representable,
      and $X$ admits an \bfem{$m$-atlas} (a smooth surjective
        morphism $\coprod U_i \rightarrow X$ from a disjoint union of
        affine derived schemes such that each $U_i \rightarrow X$ is
      $(m-1)$-representable).
    \item A morphism $f: X \rightarrow Y$ is \bfem{$m$-representable}
      if for any affine derived scheme $U \rightarrow Y$, the
      homotopy pullback $X \times^h_Y U$ is an $m$-geometric stack.
  \end{enumerate}
\end{definition}

\begin{definition}\cite[Definition 1.3.3.1]{ToenHAG}
  An object $X \in dSt_{\K}$ is called a \bfem{derived Artin
  $\K$-stack} if it is $m$-geometric for some integer $m \geq 0$, and
  the underlying classical stack is 1-truncated.
\end{definition}

The upshot is that any such object $ X$ of $ dSt_{\K}$ comes with a
smooth surjective morphism  $\varphi: U \rightarrow  X$ with $U$ a
derived $\K$-scheme. We call such morphism an \bfem{atlas.}

\subsection{Perfect complexes and key properties}

Recall that we define the \bfem{algebra of functions} on $X \in dSt_{\K}$ by
\begin{equation}\label{defn_O(X) as inf limit}
  \mathcal{O}(X)= \displaystyle \lim_{A \in cdga_{\K}^{\leq 0},
  \ \spec A \rightarrow X} A.
\end{equation} Likewise,  the \bfem{$\infty$-category of
quasi-coherent sheaves on $X$} is defined to be
\begin{equation}\label{defn_QCoh(X) as inf limit}
  {QCoh}(X)= \displaystyle \lim_{A \in cdga_{\K}^{\leq 0}, \ \spec A
  \rightarrow X} {{\rm Mod}_A},
\end{equation}
where ${\rm Mod}_A$ denotes the \bfem{$\infty$-category of of (unbounded)
dg $A$-modules}\footnote{The homotopy category of ${\rm Mod}_A$ is equivalent to the unbounded derived
category of $A$-modules.}. Within the $\infty$-category of quasi-coherent sheaves $QCoh(X)$, a
particularly well-behaved class is that of \emph{perfect complexes}.
The properties of perfect modules and relative cotangent complexes
were established in the homotopical algebraic geometry framework of
To\"en and Vezzosi \cite{ToenHAG}. We extract and recall the relevant
statements here to support our main proofs in the text.

An $A$-module $M$ over a derived affine scheme $\spec A$ is called
\bfem{perfect} if it is strongly dualizable. Perfect modules satisfy
several important finiteness and descent properties:

\begin{theorem} \label{thm:HAG_perfect_properties} \emph{(Properties
  of Perfect Modules)}
  Let $X$ be a derived Artin $\K$-stack locally of finite
  presentation and $\mathcal{E}$ an object in $QCoh(X)$.
  \begin{enumerate}
    \item \label{item:perfect_fp} \emph{(Finitely presented vs.
      Perfect)} In the stable $\infty$-category of quasi-coherent
      modules, finitely presented objects are exactly the perfect
      complexes, and equivalently the retracts of finite cell
      objects \emph{\cite[Corollary 1.2.3.8]{ToenHAG}}. The last
      description makes perfect complexes visibly stable under
      shifts, retracts, fibers and cofibers; they are also stable
      under tensor products \emph{\cite[Proposition
      1.2.3.7]{ToenHAG}}.
    \item \label{item:perfect_stack} \emph{(Stack property)} The
      presheaf assigning to each derived affine scheme its
      $\infty$-category of perfect modules satisfies descent and
      forms a stack over the smooth (or \'{e}tale) topology
      \emph{\cite[Corollary 1.3.7.4]{ToenHAG}}.
    \item \label{item:perfect_conservative} \emph{(Conservativity)}
      For a covering family $\{u_i : \spec A_i \rightarrow \spec
      A\}_{i \in I}$, the family of base change functors
      $\{\mathbb{L}u_i^*\}_{i \in I}$ is conservative
      \emph{\cite[Corollary 1.3.2.7]{ToenHAG}}.
    \item \label{item:perfect_pointwise} \emph{(Conservativity on
      points)} A morphism $\phi: \mathcal{E} \rightarrow \mathcal{F}$
      of perfect complexes on $X$ is an equivalence if and only if
      $x^*\phi$ is an equivalence for every geometric point $x: \spec
      \K \rightarrow X$.
    \item \label{item:perfect_vanishing} \emph{(Vanishing locus)} If
      $\mathcal{E}$ is perfect on $X$, the locus of points at which
      the derived fiber of $\mathcal{E}$ vanishes is a Zariski-open
      derived substack $U \subseteq X$, and $\mathcal{E}|_U \simeq
      0$.
  \end{enumerate}
\end{theorem}

\begin{proof}
  Items (\ref{item:perfect_fp})--(\ref{item:perfect_conservative})
  are the cited results. We first prove an affine statement, from
  which (\ref{item:perfect_pointwise}) and
  (\ref{item:perfect_vanishing}) are then deduced.

  \emph{Claim.} Let $A$ be a connective cdga, finitely presented over
  $\K$, let $R := H^0(A)$, and let $K$ be a perfect $A$-dg-module.
  Then the set
  \[ Z := \{\, x \in \spec R \ : \ K \otimes_A^{\mathbb{L}} \kappa(x)
  \not\simeq 0 \,\} \]
  is Zariski closed, and $K \simeq 0$ if and only if $Z$ contains no
  closed point.

  The ring $R$ is a finitely generated $\K$-algebra
  \emph{\cite[Proposition 2.2.2.4]{ToenHAG}}; in particular it is
  noetherian and Jacobson, and by the Nullstellensatz every closed
  point of $\spec R$ has residue field $\K$.

  Set $P := K \otimes_A^{\mathbb{L}} R$. Since every residue field
  $\kappa(x)$ is an $R$-algebra,
  \[ K \otimes_A^{\mathbb{L}} \kappa(x) \simeq P
    \otimes_R^{\mathbb{L}} \kappa(x) \qquad \text{for all } x \in \spec
  R. \]
  By (\ref{item:perfect_fp}) the module $K$ is a retract of a finite
  cell $A$-module $C$, so $P$ is a retract of the finite cell
  $R$-module $C \otimes_A^{\mathbb{L}} R$. Over the discrete
  noetherian ring $R$, an induction on the cells through the long
  exact sequences shows that a finite cell module has only finitely
  many non-zero cohomology modules, each finitely generated; both
  properties pass to retracts, since $H^i(P)$ is a direct summand of
  $H^i(C \otimes_A^{\mathbb{L}} R)$. Consequently
  \[ \mathrm{Supp}(P) := \bigcup_i \mathrm{Supp}_R\, H^i(P) \]
  is a finite union of closed subsets, hence closed.

  We show $Z = \mathrm{Supp}(P)$. If $x \notin \mathrm{Supp}(P)$,
  then $H^i(P_x) = H^i(P)_x = 0$ for all $i$, so $P_x \simeq 0$ and
  therefore $P \otimes_R^{\mathbb{L}} \kappa(x) \simeq P_x
  \otimes_{R_x}^{\mathbb{L}} \kappa(x) \simeq 0$. Conversely, let $x
  \in \mathrm{Supp}(P)$ and let $m$ be the largest integer with
  $H^m(P)_x \neq 0$, which exists because only finitely many $H^i(P)$
  are non-zero. The truncation triangle
  \[ \tau^{\leq m-1} P_x \longrightarrow P_x \longrightarrow
  H^m(P)_x[-m] \]
  base changes along $R_x \rightarrow \kappa(x)$ to a triangle whose
  first term has vanishing cohomology in degrees $\geq m$, derived
  base change preserving the subcategories $D^{\leq c}$, and whose
  third term has
  \[ H^m\big(H^m(P)_x[-m] \otimes_{R_x}^{\mathbb{L}} \kappa(x)\big)
  \simeq H^m(P)_x \otimes_{R_x} \kappa(x), \]
  the degree-zero cohomology of the derived tensor product of a
  discrete module being the ordinary tensor product. The long exact
  sequence therefore gives
  \[ H^m\big(P \otimes_R^{\mathbb{L}} \kappa(x)\big) \simeq H^m(P)_x
  \otimes_{R_x} \kappa(x), \]
  which is non-zero by Nakayama's lemma, $H^m(P)_x$ being a non-zero
  finitely generated module over the noetherian local ring $R_x$.
  Thus $x \in Z$, and $Z = \mathrm{Supp}(P)$ is closed.

  Since $R$ is Jacobson, the closed set $Z$ is empty as soon as it
  contains no closed point; in that case every $H^i(P)$ has empty
  support, so $P \simeq 0$. It remains to check that $P \simeq 0$
  forces $K \simeq 0$. Each cell $A[n]$ of $C$ has vanishing
  cohomology above degree $-n$, so $K$, being a retract of $C$, is
  cohomologically bounded above. If $K \not\simeq 0$, let $m$ be the
  largest integer with $H^m(K) \neq 0$. Applying $- \otimes_A^{\mathbb{L}} R$
  to the truncation triangle $\tau^{\leq m-1} K \rightarrow K
  \rightarrow H^m(K)[-m]$ and arguing exactly as above, now along the
  connective morphism $A \rightarrow R$, gives
  \[ H^m(P) \simeq H^m(K) \otimes_{H^0(A)} R = H^m(K) \neq 0, \]
  a contradiction. This proves the Claim; the converse direction of
  its last assertion is immediate, since $Z = \emptyset$ when $K
  \simeq 0$.

  For (\ref{item:perfect_pointwise}), replace $\phi$ by $K :=
  \mathrm{cone}(\phi)$, which is perfect by (\ref{item:perfect_fp})
  and commutes with pullback, so that $\phi$ is an equivalence if and
  only if $K \simeq 0$. Choose a smooth atlas $\{u_i: \spec A_i
  \rightarrow X\}_i$; each $A_i$ is finitely presented over $\K$,
  smooth morphisms being locally of finite presentation and $X$ being
  locally of finite presentation over $\K$. A geometric point of
  $\spec A_i$ is the same as a $\K$-point of $H^0(A_i)$, i.e. a
  closed point of $\spec H^0(A_i)$, and it maps to a geometric point
  of $X$. The hypothesis therefore says that no closed point of
  $\spec H^0(A_i)$ lies in the set $Z$ of the Claim for $u_i^* K$,
  whence $u_i^* K \simeq 0$ for every $i$; by descent
  (\ref{item:perfect_stack}) and conservativity
  (\ref{item:perfect_conservative}) along the atlas, $K \simeq 0$.

  For (\ref{item:perfect_vanishing}), note that for a field extension
  $\kappa(x) \hookrightarrow L$ one has $\mathcal{E}
  \otimes^{\mathbb{L}} L \simeq (\mathcal{E} \otimes^{\mathbb{L}}
  \kappa(x)) \otimes_{\kappa(x)} L$, which vanishes if and only if
  $\mathcal{E} \otimes^{\mathbb{L}} \kappa(x)$ does. The condition
  defining $Z$ is therefore independent of the representative of a
  point and compatible with arbitrary base change, so the closed
  subsets $Z_i \subseteq \spec H^0(A_i)$ given by the Claim are
  compatible on overlaps and define a closed subset of the underlying
  topological space of $X$; its complement defines the open derived
  substack $U \subseteq X$. Every derived fiber of $\mathcal{E}|_U$
  vanishes, so applying the Claim on a smooth affine atlas of $U$,
  which is again locally of finite presentation over $\K$, gives
  $\mathcal{E}|_U \simeq 0$.
\end{proof}

\subsection{Cotangent and de Rham complexes}
\paragraph{The relative cotangent complex.} The deformation theory of
derived stacks and the definition of shifted differential forms rely
on the \textit{relative cotangent complex}. For a given morphism $A
\rightarrow B$ of cdgas, the \bfem{relative cotangent complex}
$\mathbb{L}_{B/A} \in \mathrm{Ho}(B-\mathrm{Mod})$ universally
corepresents the space of derived derivations. Globalizing this
concept, any morphism of derived Artin stacks $f: X \rightarrow Y$
possesses a relative cotangent complex $\mathbb{L}_{X/Y} \in QCoh(X)$.
This construction satisfies several fundamental properties:

\begin{theorem} \label{thm:hag_cotangent_properties}
  \emph{(Properties of Cotangent Complexes, \cite{ToenHAG})}
  \begin{enumerate}
    \item \label{item:cotangent_fp} \bfem{Finite presentation}
      \emph{\cite[Proposition 2.2.2.4]{ToenHAG}}: If $X$ and $Y$ are
      locally finitely presented derived Artin $\K$-stacks and
      $f:X\rightarrow Y$ is locally of finite presentation, then
      $\mathbb{L}_{X/Y}$ is a coherent object (almost perfect
      complex) in $QCoh(X)$. If, additionally, $X$ and $Y$ are smooth
      (or the morphism is perfect), then $\mathbb{L}_{X/Y}$ is perfect.
    \item \label{item:cotangent_transitivity} \bfem{Transitivity
      sequence} \emph{\cite[Lemma 1.4.1.16]{ToenHAG}}: For composable
      morphisms of derived stacks $X \xrightarrow{f} Y
      \xrightarrow{g} Z$, there is a canonical exact triangle
      (homotopy cofiber sequence) of perfect tangent complexes in $QCoh(X)$:
      \begin{equation}
        f^* \mathbb{L}_{Y/Z} \longrightarrow \mathbb{L}_{X/Z}
        \longrightarrow \mathbb{L}_{X/Y} \longrightarrow f^*
        \mathbb{L}_{Y/Z}[1].
      \end{equation}
    \item \label{item:cotangent_basechange} \bfem{Base change}
      \emph{\cite[Lemma 1.4.1.16]{ToenHAG}}: For any homotopy
      pullback square of derived stacks
      \begin{equation*}
        \begin{tikzcd}
          X' \arrow[r, "g'"] \arrow[d, "f'"'] & X \arrow[d, "f"] \\
          Y' \arrow[r, "g"'] & Y
        \end{tikzcd}
      \end{equation*}
      the canonical base change morphism $(g')^*\mathbb{L}_{X/Y}
      \longrightarrow \mathbb{L}_{X'/Y'}$ is an equivalence in
      $QCoh(X')$. In particular, the derived pullback of the relative
      cotangent complex $\mathbb{L}_{X/Y}$ to a geometric fiber $X_y$
      canonically identifies with the cotangent complex of the fiber:
      $\iota_y^*\mathbb{L}_{X/Y} \simeq \mathbb{L}_{X_y}$.
  \end{enumerate}
\end{theorem}

\paragraph{The de Rham complex.} When $X=\spec A$ is a derived affine
scheme with $A$ a cofibrant cdga (e.g. standard form cdgas with the
property that $\mathbb{L}_A\simeq \Omega_A^1$), then its \emph{de
Rham complex} is given as $$DR(X)= DR(A)\simeq\mathrm{Sym}_A
(\Omega^1_A [-1]),$$ where the graded mixed differential is given by
the universal derivation $A\rightarrow \Omega^1_A$ extended via the
Leibniz rule. Then for $X$ an arbitrary stack, we define its \bfem{de
Rham complex} as
\begin{equation}\label{defn_DR(X) as inf limit}
  DR(X)= \lim_{A \in cdga, \ \spec A \rightarrow X} DR(A).
\end{equation}
Note that any derived stack $X$ has a graded mixed de Rham complex
$DR(X)$. If, in addition, $X$ is Artin and $\mathbb{L}_X$ denotes its
cotangent complex, then we have
\begin{equation}
  \mathcal{A}^p(X,n) \simeq \pi_0 Map_{QCoh(X)}(\mathcal{O}_X,
  \wedge^p\mathbb{L}_X[n]).
\end{equation}

\bibliographystyle{alpha}
\bibliography{refss}

\end{document}